\documentclass[11pt]{article}
\usepackage{amsmath,amssymb,amsthm,graphicx,color,bbm}
\usepackage[left=1in,right=1in,top=1in,bottom=1in]{geometry}
\usepackage[british]{babel}
\usepackage{hyperref} 
\usepackage{enumitem}
\usepackage{csquotes,caption}
\usepackage{subcaption}
\usepackage{tikz}
\usetikzlibrary{decorations.markings}
\usetikzlibrary{automata, positioning,arrows,arrows.meta}

\newcommand{\white}[1]{{\color{white} #1}}
 
\newsavebox\configa
\sbox\configa{\begin{tikzpicture}[domain=0:1, scale=1.3]
\draw[dotted,<->] (0,0) -- (3,0);
\draw[black,fill=white] (0.5,0) circle (.5ex);
\draw[white,fill=white] (1,0.21) circle (.5ex);
\draw[black,fill=black] (1,0) circle (.5ex);
\node at (1.5,-0.7)       { $(1,1)$};
\draw[black,fill=black] (1.5,0) circle (.5ex);
\draw[black,fill=black] (2,0) circle (.5ex);
\draw[black,fill=white] (2.5,0) circle (.5ex);
\end{tikzpicture}}

\newsavebox\configb
\sbox\configb{\begin{tikzpicture}[domain=0:1, scale=1.3]
\draw[dotted,<->] (0,0) -- (3,0);
\draw[black,fill=white] (0.5,0) circle (.5ex);
\draw[black,fill=black] (1,0) circle (.5ex);
\draw[black,fill=black] (1,0.21) circle (.5ex);
\node at (1.5,-0.7)       { $(0,2)$};
\draw[black,fill=white] (1.5,0) circle (.5ex);
\draw[black,fill=black] (2,0) circle (.5ex);
\draw[black,fill=white] (2.5,0) circle (.5ex);
\end{tikzpicture}}

\newsavebox\configc
\sbox\configc{\begin{tikzpicture}[domain=0:1, scale=1.3]
\draw[dotted,<->,white] (-0.25,0) -- (2.75,0);
\draw[dotted,<->] (0,0) -- (2.5,0);
\draw[black,fill=white] (0.5,0) circle (.5ex);
\draw[black,fill=black] (1,0) circle (.5ex);
\draw[black,fill=black] (1,0.21) circle (.5ex);
\node at (1.25,-0.7)       { $(0,1)$};
\draw[black,fill=black] (1.5,0) circle (.5ex);
\draw[black,fill=white] (2,0) circle (.5ex);
\end{tikzpicture}}

\newsavebox\configaa
\sbox\configaa{\begin{tikzpicture}[domain=0:1, scale=1.0]
\draw[dotted,white,<->] (0,0) -- (3,0);
\draw[white,fill=white] (0.5,0) circle (.5ex);
\draw[white,fill=white] (1,0.21) circle (.5ex);
\draw[white,fill=white] (1,0) circle (.5ex);
\node at (2,-0.8)       {  \white{$(1,1)$}  };
\draw[white,fill=white] (1.5,0) circle (.5ex);
\draw[white,fill=white] (2,0) circle (.5ex);
\draw[white,fill=white] (2.5,0) circle (.5ex);
\end{tikzpicture}}

\usepackage[backend=biber,citestyle=numeric,doi=false,isbn=false,url=false,eprint=false,giveninits=true,maxbibnames=10,sortcites]{biblatex}
\AtEveryBibitem{\clearfield{number}}
\AtEveryBibitem{\clearfield{series}}

\DeclareFieldFormat[article,inbook,incollection,inproceedings,patent,thesis,unpublished]{title}{{#1\isdot}}
  \renewbibmacro{in:}{%
  \ifentrytype{article}{}{\printtext{\bibstring{in}\intitlepunct}}}

\usepackage{hyperref}
\usepackage{caption}
\hypersetup{
    colorlinks=false,
    pdfborder={0 0 0},
}

\newtheorem{theorem}{Theorem}[section]
\newtheorem{lemma}[theorem]{Lemma}
\newtheorem{proposition}[theorem]{Proposition}
\newtheorem{corollary}[theorem]{Corollary}

\theoremstyle{definition}
\newtheorem{definition}[theorem]{Definition}
\newtheorem{examplex}[theorem]{Example}

\newtheorem{problem}[theorem]{Problem}

\newenvironment{example}
  {\pushQED{\qed}\examplex}
  {\popQED\endexamplex}

\theoremstyle{remark}
\newtheorem{remark}[theorem]{Remark}
\newtheorem{remarks}[theorem]{Remarks}

\numberwithin{equation}{section}

 \allowdisplaybreaks

\newcommand{\1}[1]{{\mathbbm{1}\mkern -1.5mu}{\left\{#1\right\}}}
\newcommand{\2}[1]{{\mathbbm{1}}_{#1}}

\newcommand{\R}{{\mathbb R}}

\newcommand{\Z}{{\mathbb Z}}
\newcommand{\N}{{\mathbb N}}

\newcommand{\ZP}{{\mathbb Z}_+}
\newcommand{\RP}{{\mathbb R}_+}

\newcommand{\tX}{{\widetilde X}}

\DeclareMathOperator{\Exp}{\mathbb{E}}
\let\Pr\relax
\DeclareMathOperator{\Pr}{\mathbb{P}}

\newcommand{\eps}{\varepsilon}

\newcommand{\re}{{\mathrm{e}}}
\newcommand{\rc}{{\textup{c}}}
\newcommand{\ud}{{\mathrm d}}

\newcommand{\cF}{{\mathcal F}}
\newcommand{\cG}{{\mathcal G}}

\newcommand{\cL}{{\mathcal L}}

\newcommand{\cP}{{\mathcal P}}

\newcommand{\fB}{{\mathfrak{B}}}
\newcommand{\fC}{{\mathfrak{C}}}

\newcommand{\as}{\ \text{a.s.}}

\newcommand{\ou}{\overline{u}}

\newcommand{\bigmid}{\; \bigl| \;}
\newcommand{\Bigmid}{\; \Bigl| \;}
\newcommand{\biggmid}{\; \biggl| \;}

\newcommand{\IP}{{\mathbb P}}
\newcommand{\IE}{{\mathbb E}}

\newcommand{\bbX}{{\mathbb X}}
\newcommand{\ee}{{\mathbf e}}

\newcommand{\ones}{\mathbf{1}}
  
\makeatletter
\def\namedlabel#1#2{\begingroup  
    (#2)%
    \def\@currentlabel{#2}%
    \phantomsection\label{#1}\endgroup
}
\makeatother

\newlist{myenumi}{enumerate}{10}
\setlist[myenumi]{leftmargin=0pt, labelindent=\parindent, listparindent=\parindent, labelwidth=0pt, itemindent=!, itemsep=1pt, parsep=4pt}

\newlist{thmenumi}{enumerate}{10}
\setlist[thmenumi]{leftmargin=0pt, labelindent=\parindent, listparindent=\parindent, labelwidth=0pt, itemindent=!}

\title{One-dimensional particle clouds with elastic collisions}

\author{Mikhail Menshikov\footnote{\raggedright Department of Mathematical Sciences, Durham University, Durham, UK. \href{mailto:mikhail.menshikov@durham.ac.uk}{\texttt{mikhail.menshikov@durham.ac.uk}}, \href{mailto:andrew.wade@durham.ac.uk}{\texttt{andrew.wade@durham.ac.uk}}.} \and Serguei Popov\footnote{Centro de Matem\'atica, University of Porto, Porto, Portugal. \href{mailto:serguei.popov@fc.up.pt}{\texttt{serguei.popov@fc.up.pt}}.} \and Andrew Wade\footnotemark[1]}

\date{\today}

\begin{document}
\maketitle

\begin{abstract}
We study an interacting particle system of a finite number of labelled particles on the integer lattice, in which particles have intrinsic masses and left/right jump rates. If a particle is the minimal-label particle at its site when it tries to jump left, the jump is executed. If not, `momentum' is transferred to increase the rate of jumping left
   of the minimal-label particle.    Similarly for jumps to the right. 
   The collision rule is `elastic' in the sense that the net rate of flow of  mass
 is independent of the present configuration, in contrast to the exclusion process, for example. 
 We show that the particle masses and jump rates determine explicitly,
 via
  a concave majorant of a simple `potential' function associated to the masses and jump rates, 
 a unique partition of the system into maximal stable subsystems.
 The internal configuration of each stable subsystem remains tight, while the location of each stable subsystem obeys a strong law of large numbers with an explicit speed. We indicate connections to adjacent models, including diffusions with rank-based coefficients, where, to the authors' knowledge, the problem of identifying the analogous  decomposition into stable sub-systems is yet to be fully solved.
   \end{abstract}

\medskip


\medskip


\section{Definitions and main results}
\label{sec:definitions}

We consider dynamics of an interacting system of~$N$ ordered particles 
performing continuous-time 
nearest-neighbour random walks
on the integer lattice~$\Z$
with \emph{elastic collisions}. 
Each particle $i \in \{1,2,\ldots, N\}$ (labelled left to right)
is endowed with intrinsic jump rates $a_i, b_i \in \RP := [0,\infty)$
and an intrinsic \emph{mass}~$m_i \in (0,\infty)$. 

The dynamics of the system are as follows.
If the site occupied by particle~$i$ is occupied by no other particle,
then particle $i$ jumps to the left with rate~$a_i$ 
and to the right with rate~$b_i$, independently of the other particles;
hence particles perform independent continuous-time random walks
as long as they avoid each other. When two or more particles
occupy the same site, only the particle with the smallest label in the stack can jump left, and only the particle
with the largest label in the stack can jump right, and the rates are modified
by what we call the elastic collision rule. 
Specifically, if particles $k,\ldots,k+\ell$ (and no other) are at a given site
 then
 particle~$k$ jumps to the left with rate
 $m_k^{-1} \sum_{j=k}^{k+\ell} m_{j} a_{j}$,
while particle  $k+\ell$ jumps
to the right with rate 
$m^{-1}_{k+\ell} \sum_{j=k}^{k+\ell} m_{j} b_{j}$.

Denote by $X_i (t) \in \Z$ the position of particle~$i$ at time~$t \in \RP$,
with initial configuration
$X_1 (0) \leq X_2 (0) \leq \cdots \leq X_N(0)$; 
observe that the collision rule preserves the weak
order~$X_1 (t) \leq X_2 (t) \leq \cdots \leq X_N (t)$ for all $t \in \RP$. 
We give a slightly more formal definition shortly, and refer to Figure~\ref{fig:elastic} for a schematic.

\begin{figure}[t]
\centering
\scalebox{0.75}{
 \begin{tikzpicture}[domain=0:1, scale=1.5]
\draw[dotted,<->] (0,0) -- (8,0);
\node at (8.4,0)       {$\Z$};
\draw[black,fill=white] (1,0) circle (.5ex);
\draw[black,fill=black] (2,0) circle (.5ex);
\draw[black,fill=black] (3,-0.1) circle (.5ex);
\draw[black,fill=black] (3,0.1) circle (.5ex);
\draw[black,fill=white] (4,0) circle (.5ex);
\draw[black,fill=white] (5,0) circle (.5ex);
\draw[black,fill=black] (6,0) circle (.5ex);
\draw[black,fill=black] (6,0.2) circle (.5ex);
\draw[black,fill=black] (6,-0.2) circle (.5ex);
\draw[black,fill=white] (7,0) circle (.5ex);
\node at (1.3,-0.6)       {\small rate $a_1$};
\node at (2.5,0.6)       {\small $b_1$};
\draw[black,-{Stealth[scale=2.0]}] (2,0) arc (330:221:0.58);
\draw[black,-{Stealth[scale=2.0]}] (3,-0.1) arc (320:218:0.6);
\node at (2.5,-0.6)       {\small $a_2+a_3$};
\node at (3.5,0.6)       {\small $b_2+b_3$};
\draw[black,-{Stealth[scale=2.0]}] (6,-0.2) arc (310:217:0.65);
\node at (5.5,-0.6)       {\small $a_4+a_5+a_6$};
\node at (6.5,0.6)       {\small $b_4+b_5+b_6$};
\draw[black,-{Stealth[scale=2.0]}] (2,0) arc (150:39:0.58);
\draw[black,-{Stealth[scale=2.0]}] (3,0.1) arc (140:36:0.6);
\draw[black,-{Stealth[scale=2.0]}] (6,0.2) arc (130:35:0.65);
\end{tikzpicture}}
\caption{Schematic for the $N=6$ elastic particle system with identical masses $m_i \equiv m \in (0,\infty)$ for all $i$. Pictured is the configuration $X_1 (t) = 0$, $X_2 (t) = X_3 (t) = 1$, $X_4 (t) = X_5 (t) = X_6(t) = 4$ and the transition rates from this configuration are indicated on the arrows. In the case of multiple particles occupying the same site, we imagine that particles are stacked in increasing order of index, and it is the bottom and top particles that are allowed to move: the transition to the left would move the particle from the base of the stack, while the transition to the right would move the particle from the top of the stack. To contrast the elastic dynamics with exclusion dynamics, we refer to Figure~\ref{fig:exclusion} below; note that in the identical-mass setting, the total activity rate of the elastic system is $\sum_i (a_i+b_i)$ independently of the current configuration.}
\label{fig:elastic}
\end{figure}
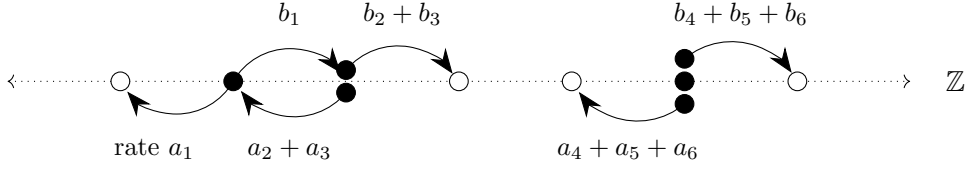

We use the word `\emph{elastic}' to describe the distinguishing property of this interacting-particle system 
that \emph{the total momentum is independent of the present configuration},
where momentum is the (net) rate at which mass moves to the right 
(see Lemma~\ref{lem:no-mass-random-walk} below and surrounding discussion). 
This model is equivalent to one proposed
in~\S6.2~of~\cite{mmpw} 
as a lattice particle system
possessing the elastic collision property of
the continuum \emph{Atlas model}, in contrast to the classical
simple \emph{exclusion process} which does not have the elastic property, since the exclusion
dynamics suppresses movement of particles in some configurations.
See~\S\ref{sec:equivalent}
for a description of the model of~\cite{mmpw} and its equivalence
to the one here, and for background on the exclusion process and the Atlas model. Of the two main results in this paper, Theorems~\ref{thm:stability_condition} and~\ref{thm:cloud_decomposition}, analogues of both in the exclusion context are given by~\cite{mmpw}, while for the Atlas model, only the analogue of Theorem~\ref{thm:stability_condition} is known for the general-parameter case (see \S\ref{sec:atlas} for further comment).

To give a physical motivation to the elastic collision
rule, note that, first, if we interpret transitions
rates as speeds, when it does not share occupancy with any other particle, particle~$i$ has `momentum
to the left' $m_ia_i$ and `momentum to the right'  $m_ib_i$.
On the other hand, if particle~$i$ is part of a stack of multiple particles,  the momentum generated by the stack to the 
left and right is obtained by summing the individual
momenta of the particles in the stack. For example, if particles $k, k+1, \ldots, k+\ell$ are the occupants of a particular site,
then the elastic collision rule imparts to particle~$k$ a `momentum to the left' of
$m_k \cdot m_k^{-1} \sum_{j=k}^{k+\ell} m_{j} a_{j} = \sum_{j=k}^{k+\ell} m_{j} a_{j}$,
while particles $k+1, \ldots, k+\ell$ have `momentum to the left' of zero. Similarly to the right. In this sense, `momentum' is conserved. 

We introduce some more notation and describe our main results. For $n \in \N := \{1,2,3,\ldots\}$, define $[n] := \{ i \in \Z : 1 \leq i \leq n\}$.
The configuration space of the system of $N \in \N$ particles is
$\bbX_N\subset \Z^N$, given by
\begin{equation}
\label{eq:configuration-space}
\bbX_N = \big\{ (x_i )_{i \in [N]} \in \Z^N 
: x_1 \leq \cdots \leq x_N \big\}. 
\end{equation}
We write $X(t) \in \bbX_N$ for the state at time $t \in \RP$, with coordinates
$X(t) = (X_i (t))_{i \in [N]}$. 
We start the system from a deterministic (but arbitrary\footnote{Random initial configurations can also be considered, and would not change the almost-sure asymptotic statements in the main results.}) initial state $X (0) \in \bbX_N$. 
The process $X := (X(t))_{t \in \RP}$ is a continuous-time Markov chain
on the countable state space~$\bbX_N$, in which $X_i(t)$ jumps to $X_i(t)-1$ at rate~$A_i (X(t))$
and  $X_i(t)$ jumps to $X_i(t)+1$ at rate $B_i (X(t))$ where, for $x = (x_i)_{i \in [N]} \in \bbX_N$,
\begin{equation}
    \label{eq:rates}
    A_i (x) := \2 {\{ x_{i-1} < x_i \}} \sum_{j = i}^N \frac{m_ja_j}{m_i} \2 {\{ x_j  = x_i \}},
    ~~~
    B_i (x) :=  \2 {\{ x_{i+1}  > x_i  \}} \sum_{j = 1}^i \frac{m_jb_j}{m_i} \2 {\{ x_j  = x_i  \}} ,
    \end{equation}
with conventions $x_0 = -\infty$ and $x_{N+1}= +\infty$.
Denote the intrinsic velocities of the particles 
\begin{equation}
    \label{eq:intrinsic-speed}
    u := (u_i)_{i \in [N]}, \text{ where } u_i := b_i - a_i, \text{ for } i \in [N].
\end{equation}
We then define $U_0:=0$, $M_0:= 0$, and, for $k \in [N]$,
\begin{equation}
    \label{eq:U-M-def}
    U_k := - \sum_{i=1}^k m_i u_i , \text{ and } 
    M_k := \sum_{i=1}^k m_i .
\end{equation}
Also, for $\ell \in \{0,1,\ldots, k\}$ define
\begin{equation}
    \label{eq:U-M-diff-def}
    U_{\ell,k} := U_k - U_\ell = - \sum_{i=\ell+1}^k m_i u_i , \text{ and } 
    M_{\ell,k} := M_k - M_\ell = \sum_{i=\ell+ 1}^k m_i .
\end{equation}

\begin{remark}
\label{rem:zeros}
Because of the elastic interaction mechanism,
many of the $a_i$s and $b_i$s can be set to zero and the model still be non-trivial. For example, one can take $b_1 >0$, $a_N >0$, and all other $a_i, b_i = 0$, and the elastic process still, with positive probability, moves all particles any finite distance to the left or right, while, in contrast, the exclusion process with the same parameters (see~\S\ref{sec:exclusion}) would reach an absorbing state in finitely many steps.
\end{remark}

\begin{remark}
\label{rem:potential}
    The reason for the negative sign in the definition of~$U_k$ in~\eqref{eq:U-M-def} is to conform with the interpretation of $U_k/M_k$ as a sort of \emph{potential} function associated with the system. Roughly speaking, the evolution of the system will favour moving particles in directions of negative gradient of the potential; hence the choice of sign. We elaborate on this following Definition~\ref{def:majorant} below.
\end{remark}

For non-empty $\fC \subseteq [N]$ denote the span of particles labelled by~$\fC$ as
\begin{equation}
\label{eq:radius-def}
\Delta_\fC (t) :=  \sup_{i \in \fC} X_i (t) - \inf_{i \in \fC} X_i (t)
= X_{\max \fC} (t) - X_{\min \fC} (t), \text{ for } t \in \RP.
\end{equation}

 Our first main result is a criterion for \emph{stability} of the system,
 meaning that 
 the system  evolves as a single \emph{cloud} of particles, 
within which the inter-particle distances remain exponentially tight, with a single characteristic \emph{speed}. The formal statement is the following.

\begin{theorem}[Stability criterion]
\label{thm:stability_condition}
Let $N \in \N$, $m_i \in (0,\infty)$ for all $i \in [N]$,
and $a_i, b_i \in \RP$ for all $i \in [N]$.
With $U_k, M_k$ defined at~\eqref{eq:U-M-def}, suppose that
\begin{equation}
\label{eq:stability_condition}
\frac{U_k}{M_k}<\frac{U_N}{M_N} , \text{ for all } k \in [N-1].
\end{equation}
Then the system is a single stable cloud in the following sense.
\begin{thmenumi}[label=(\alph*)]
    \item 
    \label{thm:stability_condition-a} 
    {\bf {Limiting speed.}}
For every $i \in [N]$, there holds the strong law of large numbers 
    \begin{equation}
        \label{eq:stable-speed}
 \lim_{t \to \infty} t^{-1} {X_i (t)} = - \frac{U_{N}}{M_{N}}, \as     \end{equation}
    \item 
    \label{thm:stability_condition-b} 
    {\bf {Relative stability.}} 
      There exist constants $C \in \RP$ and $\delta > 0$ (depending on the $a_i, b_i$, and $m_i$)
    such that 
 \[
\sup_{t \in \RP} \Pr \left[ \Delta_{[N]} (t) \geq s \right] \leq  C \Bigl[ 1 + \re^{C \Delta_{[N]} (0) } \Bigr] \re^{-\delta s} , \text{ for all }  s \in \RP,      \]
    and, moreover,
    \[ \limsup_{t \to \infty} \frac{\Delta_{[N]} (t)}{\log t} < \infty , \as\]
\end{thmenumi}
Furthermore, condition~\eqref{eq:stability_condition} is necessary for~\ref{thm:stability_condition-b} to hold.
 \end{theorem}

\begin{remark}
    \label{rem:stability_condition}
    That condition~\eqref{eq:stability_condition} is sufficient for relative stability was conjectured in~\cite{mmpw} (for a modified but equivalent version of the process, as explained in~\S\ref{sec:expanded}), and is a lattice analogue of the stability result for Atlas-type models (e.g.~Theorem~8 of~\cite{pp}). 
    That condition~\eqref{eq:stability_condition} is necessary for relative stability~\ref{thm:stability_condition-b} is a consequence of the more general partial stability result, Theorem~\ref{thm:cloud_decomposition}, that we present below, and which, to the authors' knowledge, as yet has no complete analogue in the setting of Atlas models. Theorem~\ref{thm:cloud_decomposition} also shows that~\ref{thm:stability_condition-a} \emph{can} hold if~\eqref{eq:stability_condition} is violated, in a system with multiple stable clouds all with the same characteristic speed, a simple example being the homogeneous-speed case in which $u_i \equiv u$ for all $i \in [N]$ (cf.~Example~\ref{ex:homogeneous} below).
\end{remark} 

For a configuration $x = (x_1, \ldots , x_N) \in \bbX_N$, 
define its vector of inter-particle distances
\begin{equation}
\label{eq:D-def}
D (x) := (D_i(x))_{i \in [N-1]} \in \ZP^{N-1}, \text{ where } D_i (x) := x_{i+1} - x_i  \text{ for } i \in [N-1].
\end{equation}
Then associated to process $X$ is the process $\eta (t) := (\eta_i (t) )_{i \in [N-1]}$,
on $\ZP^{N-1}$, where
\begin{equation}
    \label{eq:eta-def}
    \eta (t) := D (X(t) ), \text{i.e., } \eta_i (t) = X_{i+1} (t) - X_i (t) \text{ for } i \in [N-1].
 \end{equation}
 Clearly $(X_1(t), \eta(t))$ defines a Markov process on $\Z \times \ZP^{N-1}$
 that contains the same information as the original process~$X$. Moreover, it is not hard to see that $\eta := (\eta(t))_{t \in \RP}$ is itself a Markov process\footnote{One way to see this is to observe that increments of $X_1$ can be recovered from increments of $\eta$: $X_1$ moves left when $\eta_1$ increases with $\eta_2$ fixed, while $X_1$ moves right when $\eta_1$ decreases with $\eta_2$ fixed.}, describing the configuration relative to the left-most particle. The relative stability part of Theorem~\ref{thm:stability_condition} has the following interpretation in terms of the process~$\eta$.

\begin{corollary}
\label{cor:stability}
The Markov chain~$\eta$ on $\ZP^{N-1}$ is positive recurrent if and only if~\eqref{eq:stability_condition} holds.
 If positive recurrent, then~$\eta$ is geometrically ergodic, meaning  there exist $C< \infty$ and $\delta >0$ such that,
 \[ \sup_{t \in \RP} \Pr \big[ \| \eta (t) \| \geq s \big] \leq  C \left[1 + \re^{C \| \eta (0) \|} \right] \re^{-\delta s} , \text{ for all } s \in \RP.\]
\end{corollary}

The following example covers systems with few particles; 
a contrast with the exclusion process described in \S\ref{sec:exclusion} is given in Example~\ref{ex:not-exclusion} below.

\begin{example}[Small systems]
\label{ex:small-systems}
The case $N=1$ is trivial; then~\eqref{eq:stability_condition} holds vacuously, $U_1 = m_1( a_1 -b_1)$ and~\eqref{eq:stable-speed} reduces to
the ordinary strong law for a single random walker, $\lim_{t \to \infty} t^{-1} X_1 (t) = u_1$, a.s. When $N=2$, the stability criterion~\eqref{eq:stability_condition} reduces to the simple condition $u_1 > u_2$ (regardless of the masses), meaning that if the two particles did not interact, the leftmost would  overtake the rightmost.
For $N=3$, the masses of the particles enter, and the stability criterion~\eqref{eq:stability_condition} is 
\[ ( m_2 + m_3 ) u_1 > m_2 u_2 + m_3 u_3 \text{ and } m_1 u_1 + m_2 u_2 > (m_1 + m_2 ) u_3 ; \]
note that a consequence of the above two inequalities is that $u_1 > u_3$,
and a sufficient condition for stability is $u_1 > u_2 > u_3$.
\end{example}

Before moving on to the general case in which~\eqref{eq:stability_condition} is not satisfied, 
we give some intuition behind the 
speed $-U_N/M_N$ appearing in~\eqref{eq:stable-speed},
and the origin of condition~\eqref{eq:stability_condition}. Let 
\begin{equation}
    \label{eq:centre-of-mass}
    G:= (G(t))_{t \in \RP}, \text{ where }
G(t) := \frac{1}{M_N} \sum_{i \in [N]} m_i X_i (t), \text{ for } t \in \RP,  \end{equation} 
the \emph{centre of mass} process associated with the particle system.
In general~$G$ is not itself a Markov process, since the transition law of~$G(t)$
depends on the whole configuration $X(t)$. Nevertheless,
we show below (see Theorem~\ref{thm:mass-speed}) that there is a well-defined local \emph{speed} of~$G(t)$ which is 
always equal to $-U_N/M_N$. Moreover, in the equal-mass case where $m_i \equiv m$ independently of $i$, it is immediate from~\eqref{eq:rates} to verify the following  stronger fact, which says that~$G$ \emph{is} Markov. This is a discrete analogue of a similar observation
for diffusion systems (e.g.~Lemma~7 of~\cite{pp}, and~\S\ref{sec:atlas} below).

\begin{lemma}
\label{lem:no-mass-random-walk}
 If $m_i \equiv m \in (0,\infty)$ for all $i \in [N]$, then $G$~performs a continuous-time, spatially homogeneous random walk on $N^{-1} \Z$
      which  jumps $-m/N$ at rate $\sum_{i=1}^N a_i$ and
    $+m/N$ at rate $\sum_{i=1}^N b_i$.
\end{lemma}

In particular, under the hypotheses of Lemma~\ref{lem:no-mass-random-walk},  the strong law for Poisson processes shows that, a.s., $\lim_{t \to \infty} t^{-1} G(t) = - U_N / M_N$. We show below the same is true more generally (see~Theorem~\ref{thm:mass-speed}); roughly, 
this follows from the fact that when all the particles are at different sites, the momentum is $-U_N$, and the dynamics are designed to conserve momentum. 
This fact is true regardless of the stability or otherwise of the particle system, but it does explain why, in Theorem~\ref{thm:stability_condition}, when the system is stable, the limiting speed of the cloud as given by~\eqref{eq:stable-speed} has to be $-U_N/M_N$.

To state our general stability result, Theorem~\ref{thm:cloud_decomposition} below, we need some additional definitions.

\begin{definition}[Concave majorant, boundary, slopes]
\label{def:majorant}
\begin{myenumi}[label=(\roman*)]
    \item\label{def:majorant-i}
Given $M_0, \ldots, M_N$ and $U_0, \ldots, U_M$,
a continuous, concave function $u: [0,M_N] \to \R$ is a 
\emph{concave majorant} if $U_k \leq u(M_k)$ for every $k \in \{0,1,\ldots,N\}$. 
The (unique) \emph{least concave majorant} $\ou : [0,M_N] \to \R$ is a concave majorant  
such that $\ou (x) \leq u(x)$ for all $x \in [0,M_N]$ and every concave majorant~$u$. See Figure~\ref{fig:elastic_stable_m} for a picture.

\item\label{def:majorant-ii}
It is not hard to see that $\ou$ is piecewise linear. Denote
\begin{equation}
\label{set_of_boundary_indices}
V := \bigl\{ k \in \{0,1,\ldots, N\} : U_k = \ou(M_k) \bigr\},
\end{equation}
the 
set of \emph{boundary} indices for  $\ou$; note
that it is always the case that $\{0, N\} \subseteq V$. In other words, $V = \{ k_0, k_1, \ldots, k_\nu \}$ ($\nu \in [N]$)
is the unique set with $0 = k_0 < \cdots < k_\nu = N$, 
\begin{align}
    \label{eq:boundary-1}
    \frac{U_{k_{j-1},k_{j}}}{M_{k_{j-1},k_{j}}} & \geq  \frac{U_{k_j,k_{j+1}}}{M_{k_j,k_{j+1}}} , 
     \text{ for all } j \in [\nu-1], 
    \text{ and } \\
    \label{eq:boundary-2}
          \frac{U_{k_j,m}}{M_{k_j,m}} & < \frac{U_{k_j,k_{j+1}}}{M_{k_j,k_{j+1}}} \text{ for all } 0 \leq j < \nu \text{ and all } k_j < m < k_{j+1} ,
\end{align}
where the convention is $U_0/M_0 := \infty$.

\item\label{def:majorant-iii}
Given boundary indices $V = \{ k_0, k_1, \ldots, k_{\nu} \}$, denote the sequence of \emph{slopes} of successive boundary segments by
\begin{equation}
    \label{eq:slopes}
    v_j :=  \frac{U_{k_{j-1},k_{j}}}{M_{k_{j-1},k_{j}}} , \text{ for } j \in [\nu];
\end{equation}
then~\eqref{eq:boundary-1} says that $v_{j} \geq v_{j+1}$ for all $j \in [\nu-1]$.
\end{myenumi}
\end{definition}

\begin{remark}
    If in~\eqref{eq:boundary-1}, we demand \emph{strict} inequality,  then we obtain the set of \emph{vertices} of the least concave majorant, and the line segments between successive vertices are its \emph{faces} which have strictly ordered slopes. The set $V$ from Definition~\ref{def:majorant} can also include non-vertex boundary points, in cases where the path $(M_k, V_k)$ touches its concave majorant at the interior of a face, corresponding to equality in~\eqref{eq:boundary-1}. 
\end{remark}
 
\begin{figure}
\begin{center}
\includegraphics{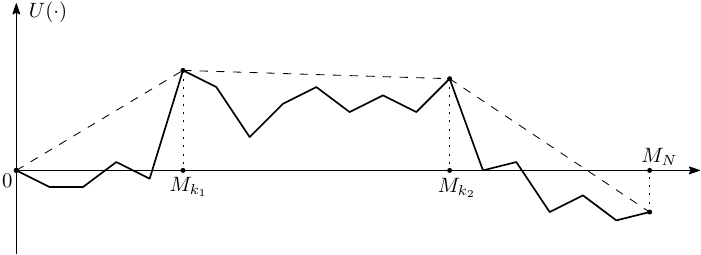}
\caption{An example of the concave majorant of the potential function $U$, identifying the stable clouds in Theorem~\ref{thm:cloud_decomposition}: here, these are
$\fC_1 = \{1,\ldots,k_1\}$, $\fC_2 = \{k_1+1,\ldots,k_2\}$, and $\fC_3 = \{k_2+1,\ldots,N\}$
(with the associated speeds being equal to minus the corresponding
slopes).}
\label{fig:elastic_stable_m}
\end{center}
\end{figure}

The next result will demonstrate a \emph{cloud decomposition}, which means 
 an ordered partition $\fC_1 < \cdots < \fC_\nu$ of $[N]$, where the \emph{clouds} $\fC_k$ are non-empty,
pairwise disjoint, have union~$[N]$, and 
where $\fC < \fC'$ for $\fC, \fC' \subseteq [N]$ means that  $i < j$ for every $i \in \fC$ and $j \in \fC'$.
Then for every $i \in [N]$ there is a unique $j \in [\nu]$ 
    such that $i \in \fC_j$; we say that particle $i$ belongs to cloud~$j$.
The next result shows that the particles in each cloud are typically close,
in the sense that all inter-particle distances within the same cloud remain exponentially tight,
but different clouds  either diverge ballistically (if their speeds differ) or
are very often well-separated (if they have the same speed, see Remark~\ref{rems:cloud_decomposition}\ref{rems:cloud_decomposition-ii}). For this last part, we will need also the following non-degeneracy condition:
    \begin{equation}
        \label{eq:non-degeneracy}
        \sum_{i \in \fC_j} a_i >0 \text{ and }  \sum_{i \in \fC_j} b_i >0 \text{ for every } j \in [\nu] ,
    \end{equation} 
    which says that every cloud has some intrinsic activity in both directions.    For $j \in [\nu-1]$, we write
\begin{equation}
    \label{eq:L-def}
    L_j (t) := \min_{i \in \fC_{j+1}} X_i (t) - \max_{i' \in \fC_j} X_{i'} (t) 
    = X_{\min \fC_{j+1}} (t) - X_{\max \fC_j} (t) 
    \in \ZP, \text{ for } t \in \RP,
\end{equation}
the minimal distance between particles from clouds~$\fC_j$ and $\fC_{j+1}$ at time $t$.

\begin{theorem}[Cloud decomposition]
\label{thm:cloud_decomposition}
Let $N \in \N$, $m_i \in (0,\infty)$ for all $i \in [N]$,
and $a_i, b_i \in \RP$ for all $i \in [N]$. Suppose that~\eqref{eq:non-degeneracy} holds. 
Suppose that  $V = \{ k_0, k_1, \ldots, k_{\nu} \}$
is the set of boundary indices for the least concave majorant, as in Definition~\ref{def:majorant}. Then for the cloud decomposition
$\fC_1, \ldots , \fC_{\nu}$ given by $\fC_j := \{ i \in [N] : k_{j-1} < i \leq k_j \}$, the following hold.
\begin{thmenumi}[label=(\alph*)]
    \item\label{thm:cloud_decomposition-a}
    {\bf {Cloud speeds.}}
For every $j \in [\nu]$ and every $i \in \fC_j$, 
with $v_j$ the slope defined at~\eqref{eq:slopes}, 
    \[ \lim_{t \to \infty} t^{-1} {X_i (t)} = - v_j, \as \]
    \item\label{thm:cloud_decomposition-b}
    {\bf {Clouds are stable.}} 
     Recall that $\Delta_\fC$ is defined at~\eqref{eq:radius-def}.  
     There exist constants $C \in \RP$ and $\delta > 0$ (depending on the $a_i, b_i$, and $m_i$)
    such that  for every $j \in [\nu]$,
 \[
\sup_{t \in \RP} \Pr \left[ \Delta_{\fC_j} (t) \geq s \right] \leq C \Bigl[1  + \re^{C \Delta_{\fC_j} (0) } \Bigr] \re^{-\delta s} , \text{ for all } s \in \RP,      \] 
    and, moreover,
    \[ \limsup_{t \to \infty} \frac{\Delta_{\fC_j} (t)}{\log t} < \infty , \as\]
    \item\label{thm:cloud_decomposition-c}
    {\bf {Cloud separations.}} 
    Let $j \in [\nu -1]$. If $v_j > v_{j+1}$, then
    $\lim_{t \to \infty} t^{-1} L_j (t) = |v_{j+1} -v_j| > 0$, a.s. If $v_j = v_{j+1}$,
    then there exists $\eps \in (0,1/2)$ such that
    \begin{equation}
        \label{eq:same-speed-separation}
    \lim_{t \to \infty} \frac{1}{t^{1-\eps}} \int_0^t  \1 { L_j (s) \leq t^\eps } \ud s
     = 0, \as,
    \end{equation}
    while nevertheless
    \begin{equation}
        \label{eq:recurrence}
    \liminf_{t \to \infty} L_j (t) = 0, \as
    \end{equation}
    \end{thmenumi}
\end{theorem}

\begin{example}[Homogeneous speeds]
\label{ex:homogeneous}
    Suppose that $u_i \equiv u \in \R$
    for all $i \in [N]$. Then $U_k/M_k = - u$ for all $k \in [N]$, meaning that $V = \{1,2,\ldots,N\}$ and 
    the cloud decomposition consists of $\nu = N$ singleton ``clouds'' all with the same speed
    $-v_j = u$. The condition~\eqref{eq:non-degeneracy} holds (and hence so does the separation result in~\ref{thm:cloud_decomposition-c} above) unless $a_i = b_i =0$ for all $i \in [N]$, in which case~\eqref{eq:same-speed-separation} will fail, since the particles do not move at all. This, and similar examples, shows that the hypothesis~\eqref{eq:non-degeneracy} cannot be removed in Theorem~\ref{thm:cloud_decomposition}\ref{thm:cloud_decomposition-c}.
\end{example}

\begin{remarks}
\label{rems:cloud_decomposition}
    \begin{myenumi}[label=(\roman*)]
    \item 
     \label{rems:cloud_decomposition-i}
    Recalling Definition~\ref{def:majorant}, we have that $V = \{ 0, N\}$, i.e., $\nu=1$ and the concave majorant $\ou$ is strictly above $(M_i, U_i)$ for every $i \in [N-1]$, if and only if~\eqref{eq:stability_condition} holds. Hence Theorem~\ref{thm:stability_condition} follows from Theorem~\ref{thm:cloud_decomposition}, although we prove them in the other order.
    \item
    \label{rems:cloud_decomposition-ii}
    Statement~\eqref{eq:same-speed-separation} says that each gap between clouds
    with equal speeds is all but a vanishing proportion of the time growing at least as a small power of~$t$,  while~\eqref{eq:recurrence} says that it is, on the other hand,  recurrent. 
A natural (but difficult) question concerns recurrence or transience of $L_A (t) := (L_j (t))_{j \in A}$ for $A \subseteq [\nu-1]$ with $v_{j}=v_{j+1}$ for every $j \in A$. 
For $N=3$ particles in $\nu=3$ singleton clouds, criteria for recurrence and transience can be deduced
 from results for random walks on~$\ZP^2$~\cite{afm} as described in Example 2.18 of~\cite{mmpw} for the related model with exclusion interaction, but we leave the general case as an open problem.
    \item
        \label{rems:cloud_decomposition-iii}
        Theorem~\ref{thm:cloud_decomposition} identifies stable subsystems of a Markov chain that may not be stable as a whole; interest in 
    such \emph{partial stability} results~\cite{afsw,GM84,mmpw}
    has been stimulated by queueing theory especially, and there is a queueing
    interpretation of the present model, as we explain in~\S\ref{sec:queues} below.
    \item    \label{rems:cloud_decomposition-iv}
    In the construction of Definition~\ref{def:majorant},
    note that if we multiply all masses $m_i$ by the same positive constant, then
    the $U_{\ell,k}$, $M_{\ell,k}$ are multiplied by that same constant, hence their ratios remain the same, and so the boundary indices $\{k_0, \ldots, k_{\nu}\}$
    remain unchanged, as do the $v_j$ at~\eqref{eq:slopes}. Similarly, if we multiply all rates $a_i, b_i$ by the same positive constant, then the $U_{\ell,k}$ are multiplied but not the $M_{\ell,k}$, so the boundary indices $\{k_0, \ldots, k_{\nu}\}$
    remain unchanged once more, but the $v_i$ are all multiplied by the common rate factor. 
   Moreover, in the special case where masses $m_i \equiv m \in (0,\infty)$ are constant, 
   the boundary indices also remain the same if one \emph{adds} the same quantity to all the $a_i$'s (and/or
all the $b_i$'s), corresponding to an affine transformation $U_k \mapsto U_k + \alpha k$ which preserves the concave majorant (cf.~Figure~\ref{fig:elastic_stable_m}).
    \item    \label{rems:cloud_decomposition-v}
    Suppose that parameters $\omega := (a_i,b_i)_{i \in [N]}$ attached to the particles are determined before the dynamics begins by sampling them according to a \emph{random environment}. That is, take $(a_1,b_1)$, \ldots, $(a_N,b_N)$ to be independent draws from a law $\mathbf{P}$, and then, given the realization of the environment~$\omega$, define the Markov chain~$X$ through rates~\eqref{eq:rates}, under a law that we now call $\Pr_\omega$ to indicate the dependence on~$\omega$. Theorem~\ref{thm:cloud_decomposition} is then a \emph{quenched} result (for given~$\omega$) but also of interest is behaviour for typical~$\omega$. For simplicity, suppose that $m_i \equiv 1 \in (0,\infty)$ (constant masses). Then the cloud decomposition is determined by the concave majorant of the $N$-step random walk (under $\mathbf{P}$) $U_0, U_1, \ldots, U_N$ with increments $-u_1, \ldots, -u_N$, by~\eqref{eq:U-M-def}. A classical result on concave majorants of random walks (see~\cite{steele,arb} and references therein) says that, assuming $u_i$ has a density under $\mathbf{P}$, the expected number of slopes of the concave majorant is asymptotically equivalent to $\log N$, and hence the (random) number of clouds $\nu$ in the decomposition of Theorem~\ref{thm:cloud_decomposition} has $\mathbf{E} \nu \sim \log N$ as $N \to \infty$ as well.
\end{myenumi}
\end{remarks}

We indicate two quite broad directions for potential future work, in addition to the recurrence/transience question raised in Remark~\ref{rems:cloud_decomposition}\ref{rems:cloud_decomposition-ii}.
\begin{itemize}
    \item In contrast to the case of exclusion interaction, where
    explicit product-geometric invariant   distributions are known in the stable case (see~\cite{mmpw}), the results above in the elastic case give no explicit form for invariant measures. There are reasons to expect that, in the elastic case,   simple explicit formulas are not available generically, but only in some special cases of the parameters, that as yet remain to be classified. We make some observations in this direction in \S\ref{sec:invariant_distributions}, but leave fuller study of invariant measures as an open problem.
    \item We have dealt here with finite systems of particles with elastic collisions, but it is a natural open problem to also consider semi-infinite systems, with particles enumerated by the natural numbers, and configurations in which there are only finitely many particles at each site.  Here one expects some progress can be made by comparison with large finite systems, but also  new and rich phenomena; see e.g.~\cite{mpw-ptrf}, and references therein, for the setting of exclusion interaction (as described in~\S\ref{sec:exclusion} in the finite case).
\end{itemize}

The outline of the rest of the paper is as follows. In \S\ref{sec:equivalent} we
describe equivalent models to the model described above (including in~\S\ref{sec:expanded} the original formulation of the particle model from~\cite{mmpw} in which no site can be occupied by more than one particle). In~\S\ref{sec:exclusion} we contrast the elastic model with the simple exclusion process, which has a non-elastic collision mechanism and is one of the most studied models of interacting particle systems,
while in~\S\ref{sec:atlas} we draw parallels with a well-studied continuum model of mutually-reflecting diffusions. In~\S\ref{sec:queues} we formulate a queueing model equivalent to the elastic particle model. In the exclusion-process context, a similar translation gives a Jackson network, for which partial stability results due to~\cite{GM84} were a key component to developing the corresponding cloud decomposition in~\cite{mmpw}; as far as we know,
no such results are available for the class of queueing models we end up with here, and so we need a different approach. 
The proofs of our main results are divided between proof of, loosely speaking, stability (Theorem~\ref{thm:stability_condition}) in~\S\ref{sec:proofs-stability} and instability and hence the cloud decomposition (Theorem~\ref{thm:cloud_decomposition}) in~\S\ref{sec:proofs-instability}. The arguments for stability, including the proof of Theorem~\ref{thm:stability_condition} given in~\S\ref{sec:speed}, use some Lyapunov function ideas presented in~\S\ref{sec:lyapunov}. The most involved part of the proof of instability concerns
the case where there are several different clouds of the same speed, and here we make use of some apparently novel tools from martingale defocusing that we present in a self-contained appendix (\S\ref{sec:submartingale}). 
The proof of Theorem~\ref{thm:cloud_decomposition} is then given in~\S\ref{sec:multiple_same_speed}. 
In \S\ref{sec:invariant_distributions} we present some partial results on invariant distributions and pose some open questions in that context.

\section{Discussion of equivalent and adjacent models}
\label{sec:equivalent}

\subsection{Random walks with rank-dependent rates}
\label{sec:rank-dependent}

Here is an alternative construction of the model formulated in~\S\ref{sec:definitions} that works when all the masses are the same. 
Consider a system of $N$ labelled particles with identical masses $m_i \equiv m \in (0,\infty)$ for all $i \in [N]$. Let 
$Y(t) := (Y_1 (t), \ldots, Y_N(t)) \in \Z^N$, where 
$Y_i (t) \in \Z$ denotes the position of the particle labelled~$i$ at time $t \in \RP$,
started from $Y(0) \in \Z^N$ (not necessarily in label order). Let $\sigma_t (i)$ denote the
permutation on $[N]$ giving the time-$t$ \emph{rank} of the particle labelled $i$ among all the particles, using lexicographic order in case of ties (rank $1$ being leftmost, rank $N$ rightmost):
\[ \sigma_t (i) := \sum_{ j \in [N]} \1 { Y_j (t) < Y_i (t) } + \sum_{j \in [i]} \1 { Y_j (t) = Y_i (t)} . \]
Then define dynamics by declaring that at time~$t$,  particle~$i$ jumps $-1$ at rate $A^{\sigma_t}_i (Y(t))$ and jumps $+1$ at rate $B^{\sigma_t}_i (Y(t))$, independently of other particles, given the ranks,
where $A^\sigma$ and $B^\sigma$, for a permutation $\sigma$ with inverse $\sigma^{-1}$, are defined similarly to~\eqref{eq:rates}: for $y \in \Z^\N$, 
\begin{align*}
A^\sigma_i (y) & := \2 {\{ y_{\sigma^{-1} ( \sigma(i)-1 )} < y_{i} \}} \sum_{j = \sigma(i)}^N a_{j}  \2 {\{ y_{\sigma^{-1}(j)}  = y_{i} \}}, \\
    B^\sigma_i (y) & :=  \2 {\{ y_{\sigma^{-1} (\sigma(i)+1)}  > y_i  \}} \sum_{j = 1}^{\sigma(i)} b_j  \2 {\{ y_{\sigma^{-1}}(j)  = y_i  \}} ,
\end{align*}
with conventions $\sigma^{-1} (0) =0$, $\sigma^{-1} (N+1) = N+1$, $y_0 = -\infty$ and $y_{N+1} = +\infty$. 
Considering the \emph{ordered} particles $X_i (t) := Y_{\sigma^{-1}_t (i)} (t)$, we recover the identical-mass elastic particle system $X(t)$ defined in \S\ref{sec:definitions}. In words, the elastic collision mechanism can be obtained by starting with a system of  particles that perform independent random walks at rate parameters
determined by their ranks, and tracking the ordered configuration of particles. This is a discrete analogue of a construction of Atlas-type diffusion models (see e.g.~\cite{pp,ipbkf,kps,bfk} and~\S\ref{sec:atlas} below).

\subsection{Embedded jump chain}
\label{sec:discrete-time}

Associated to the continuous-time Markov process $X$ defined in \S\ref{sec:definitions} is the embedded jump chain,
obtained by sampling the continuous-time process at the times at which it changes configuration.
Since the jump rates of the continuous-time process are uniformly bounded above,
and uniformly bounded below if we exclude the trivial case where~$a_i = b_i \equiv 0$ for all $i \in [N]$, 
statements about stability, partial stability, exponential bounds, etc., 
from Theorems~\ref{thm:stability_condition} and~\ref{thm:cloud_decomposition} apply
equally well to the discrete-time version of the process.

The relationship between the \emph{speeds} is a little more involved.
In the particular
 case when all masses are equal, i.e., $m_i \equiv m \in (0,\infty)$ for all $i \in [N]$,
 a consequence of~\eqref{eq:rates} is that the total activity rate of the continuous-time process is constant, regardless of the configuration:
 \begin{equation}
 \label{eq:time-change}
 \lim_{h \to 0} \frac{\Pr [ X(t+h) \neq X(t) \mid X(t) = x ]}{h} = \sum_{i \in [N]} (a_i + b_i) , \text{ for every } x \in \bbX .\end{equation}
 Consequently, ergodic properties of the discrete-time and continuous-time chains
 are equivalent up to a constant multiplicative factor in terms of the right-hand side of~\eqref{eq:time-change}: this means
 that speeds are all related by the same multiplicative factor, and 
 stationary distributions coincide exactly.
In the case of general masses $m_i$, the time-change between discrete and continuous time 
is less explicit, as total activity rate depends on the ergodic behaviour internally to each stable cloud,
which is not explicitly quantified since stationary distributions are not known explicitly (see~\S\ref{sec:invariant_distributions} below).

\subsection{Expanded elastic system excluding multiple occupancy}
\label{sec:expanded}

Take the elastic process described in \S\ref{sec:definitions} in the case when all masses are equal, i.e., $m_i \equiv m \in (0,\infty)$ for all $i \in [N]$. Define $\tX_i (t) := X_i (t) + i -1$ for every $i \in [N]$ and all $t \in \RP$. The Markov process $\tX (t) := ( \tX_i (t) )_{i \in [N]}$ 
can be described as a particle system on $\Z$, with no more than one particle per site,
in which the particle at position $\tX_i$ attempts to jump left at rate $a_i$ and right at rate $b_i$.
If no particle is occupying the site of an attempted jump, the jump is executed. If a particle
is occupying the target site of an attempted jump, instead \emph{that} blocking particle immediately attempts to jump to \emph{its} neighbouring site in the same direction; and so on. In this way `momentum' is transferred to the outermost particles of contiguous blocks, as represented in Figure~\ref{fig:elastic-single-occupancy}. For the case of non-identical masses $m_i$, a similar interpretation is possible.

We call $\tX(t)$ the \emph{expanded} elastic particle system, and, when contrast is needed, we refer to the model of \S\ref{sec:definitions} as the \emph{contracted} elastic particle system. 
Up to the bijection between the two sets of configurations by the transformation
$(x_1, \ldots, x_N) \mapsto (x_1, x_2 +1, \ldots, x_N + N-1)$, the two models are equivalent. Thus Theorem~\ref{thm:stability_condition} and Theorem~\ref{thm:cloud_decomposition} apply, \emph{mutatis mutandis}, to the expanded elastic model as well. The expanded elastic model was first proposed, as far as the authors are aware, in~\S6.2 of~\cite{mmpw}, where the stability criterion in Theorem~\ref{thm:stability_condition} was conjectured.

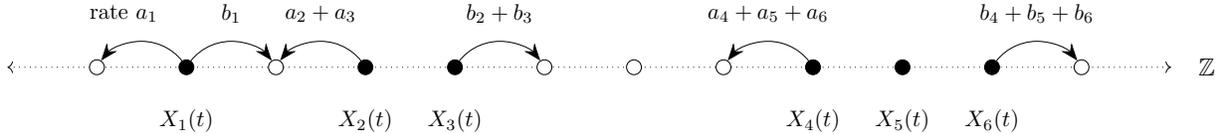
\begin{figure}[h!]
\centering
\scalebox{0.75}{
 \begin{tikzpicture}[domain=0:1, scale=1.4]
\draw[dotted,<->] (0,0) -- (13,0);
\node at (13.4,0)       {$\Z$};
\draw[black,fill=white] (1,0) circle (.5ex);
\draw[black,fill=black] (2,0) circle (.5ex);
\node at (2,-0.6)       {\small $X_1 (t)$};
\draw[black,fill=white] (3,0) circle (.5ex);
\draw[black,fill=black] (4,0) circle (.5ex);
\node at (4,-0.6)       {\small $X_2 (t)$};
\draw[black,fill=black] (5,0) circle (.5ex);
\node at (5,-0.6)       {\small $X_3 (t)$};
\draw[black,fill=white] (6,0) circle (.5ex);
\draw[black,fill=white] (7,0) circle (.5ex);
\draw[black,fill=white] (8,0) circle (.5ex);
\draw[black,fill=black] (9,0) circle (.5ex);
\node at (9,-0.6)       {\small $X_4 (t)$};
\draw[black,fill=black] (10,0) circle (.5ex);
\node at (10,-0.6)       {\small $X_5 (t)$};
\draw[black,fill=black] (11,0) circle (.5ex);
\node at (11,-0.6)       {\small $X_6 (t)$};
\draw[black,fill=white] (12,0) circle (.5ex);
\node at (1.3,0.6)       {\small rate $a_1$};
\node at (2.5,0.6)       {\small $b_1$};
\draw[black,-{Stealth[scale=1.6]}] (2,0) arc (30:141:0.58);
\draw[black,-{Stealth[scale=1.6]}] (4,0) arc (30:141:0.58);
\node at (3.5,0.6)       {\small $a_2+a_3$};
\node at (5.5,0.6)       {\small $b_2+b_3$};
\draw[black,-{Stealth[scale=1.6]}] (9,0) arc (30:141:0.58);
\node at (8.5,0.6)       {\small $a_4+a_5+a_6$};
\node at (11.5,0.6)       {\small $b_4+b_5+b_6$};
\draw[black,-{Stealth[scale=1.6]}] (2,0) arc (150:39:0.58);
\draw[black,-{Stealth[scale=1.6]}] (5,0) arc (150:39:0.58);
\draw[black,-{Stealth[scale=1.6]}] (11,0) arc (150:39:0.58);
\end{tikzpicture}}
\caption{\emph{Expanded elastic dynamics.} Take the configuration of $N=6$ equal-mass particles represented in Figure~\ref{fig:elastic}, and apply the transformation $\tX_i (t) := X_i (t) + i -1$ for every $i \in [N]$ to obtain a configuration in which the `stacks' of multiple occupancy sites are expanded into contiguous blocks of particles. This gives an equivalent version of the elastic particle system, in which no site can be occupied by more than one particle, and in which 
`momentum' of attempted jumps is transferred to the outermost particles of contiguous blocks.}
\label{fig:elastic-single-occupancy}
\end{figure}

\subsection{Simple exclusion process with particle-wise heterogeneity}
\label{sec:exclusion}

The exclusion process, originating with~\cite{spitzer70,mgp},
is one of the most intensively studied interacting particle models,
but it differs crucially from the elastic model of~\S\ref{sec:definitions} because collisions suppress activity of the system, as we will explain. The version of the exclusion process that serves as a
comparator for our model (in the case of identical masses $m_i \equiv m$) consists of $N$ particles,
with configurations in $\bbX_N$, and particle $i$ has jump rates $a_i$, $b_i$ to the left, right, when it is on its own, exactly like the model of \S\ref{sec:definitions}. The difference is that when multiple particles occupy the same site, the minimal index particle jumps left at only its intrinsic rate, rather than the sum of all rates of the particle stack. Similarly for jumps to the right;
thus the exclusion collision rule replaces the rates~$A_i, B_i$ from~\eqref{eq:rates} by simply
$ A^\textrm{exc}_i (x) := \2 {\{ x_{i-1} < x_i \}} a_i$ and $B^\textrm{exc}_i (x) :=  \2 {\{ x_{i+1}  > x_i  \}} b_i$.
See Figure~\ref{fig:exclusion} for an illustration.

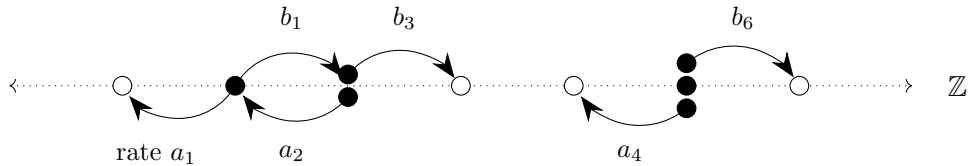
\begin{figure}[h!]
\centering
\scalebox{0.75}{
 \begin{tikzpicture}[domain=0:1, scale=1.5]
\draw[dotted,<->] (0,0) -- (8,0);
\node at (8.4,0)       {$\Z$};
\draw[black,fill=white] (1,0) circle (.5ex);
\draw[black,fill=black] (2,0) circle (.5ex);
\draw[black,fill=black] (3,-0.1) circle (.5ex);
\draw[black,fill=black] (3,0.1) circle (.5ex);
\draw[black,fill=white] (4,0) circle (.5ex);
\draw[black,fill=white] (5,0) circle (.5ex);
\draw[black,fill=black] (6,0) circle (.5ex);
\draw[black,fill=black] (6,0.2) circle (.5ex);
\draw[black,fill=black] (6,-0.2) circle (.5ex);
\draw[black,fill=white] (7,0) circle (.5ex);
\node at (1.3,-0.6)       {\small rate $a_1$};
\node at (2.5,0.6)       {\small $b_1$};
\draw[black,-{Stealth[scale=2.0]}] (2,0) arc (330:221:0.58);
\draw[black,-{Stealth[scale=2.0]}] (3,-0.1) arc (320:218:0.6);
\node at (2.5,-0.6)       {\small $a_2$};
\node at (3.5,0.6)       {\small $b_3$};
\draw[black,-{Stealth[scale=2.0]}] (6,-0.2) arc (310:217:0.65);
\node at (5.5,-0.6)       {\small $a_4$};
\node at (6.5,0.6)       {\small $b_6$};
\draw[black,-{Stealth[scale=2.0]}] (2,0) arc (150:39:0.58);
\draw[black,-{Stealth[scale=2.0]}] (3,0.1) arc (140:36:0.6);
\draw[black,-{Stealth[scale=2.0]}] (6,0.2) arc (130:35:0.65);
\end{tikzpicture}}
\caption{\emph{Contracted exclusion dynamics.} Schematic showing a configuration of $N=6$ equal-mass particles and transition rates indicated on the arrows. In the case of multiple particles occupying the same site, we imagine that they are stacked in increasing order of index, and it is the bottom and top particles that are allowed to move: the transition to the left would move the particle from the base of the stack, while the transition to the right would move the particle from the top of the stack. To translate a configuration to the more standard exclusion configuration, shift particle $i$ to the right by $i-1$ sites; this change of variables shows equivalence of this `contracted' exclusion process to the classical exclusion process in which sites are occupied by at most one particle.}
\label{fig:exclusion}
\end{figure}

 It should be noted that the exclusion process is usually described
in terms of configurations expanded using the transformation $(x_1, \ldots, x_N) \mapsto (x_1, x_2 +1, \ldots, x_N + N-1)$, but,
similarly to~\S\ref{sec:expanded}, there is an equivalence exactly as between the contracted and expanded elastic models (cf.~Figures~\ref{fig:elastic} and~\ref{fig:elastic-single-occupancy}).

\begin{example}[Small systems, continued]
\label{ex:not-exclusion}
To emphasize that there is no simple monotonicity relation between the exclusion and elastic interactions, we give examples (with $m_i \equiv m \in (0,\infty)$ constant) using the same rate parameters where (i) the model with exclusion interaction is stable but with  elastic interaction is not, and (ii) vice versa. 

For case (i), consider~$N=3$ particles with intrinsic rates $a_1 = a_2 = a_3 = b_1 =1$, $b_2 = 2$, 
and $b_3 = 1/3$.
Example~2.10 of~\cite{mmpw} shows that under exclusion dynamics, this system is stable, i.e., $\{1,2,3\}$ is a single cloud. On the other hand  Theorem~\ref{thm:cloud_decomposition} above shows that, under elastic dynamics, the stable clouds are $\{ 1\}$ and $\{2,3\}$ with corresponding speeds $-v_1 = 0$ and $-v_2 =1/6$; in the elastic case, particle 2 can push particle 3 faster to the right.

For case (ii), consider~$N=3$ particles with intrinsic rates $a_1 = b_1 = b_3 =1$, $a_2 = 1/2$, $b_2=3/2$, and $a_3=3$. 
Now Example~2.10 of~\cite{mmpw} shows that under exclusion dynamics, the stable clouds are $\{ 1\}$ and $\{2,3\}$, while Example~\ref{ex:small-systems} (or Theorem~\ref{thm:stability_condition}) shows that, under elastic dynamics, the whole system is stable. 
\end{example}

There is a bijection between the $N$-particle exclusion process and an \emph{open Jackson network} of~$N-1$ queues (see~\S3 of~\cite{mmpw} and references therein). In the work of Malyshev and the present authors~\cite{mmpw}, this bijection was used to apply results of Goodman and Massey~\cite{GM84} on partial stability for Jackson queueing networks to obtain the cloud decomposition of the exclusion system. The partial stability results (Theorems~2.1 and~2.3 of \cite{mmpw}) are analogous to
the present Theorems~\ref{thm:stability_condition} and~\ref{thm:cloud_decomposition}, but 
the algebra of cloud speeds in the exclusion/Jackson case is rather more complicated, 
since the elastic property is absent. On the other hand, the exclusion/Jackson setting
turns out to be simpler from the point of view of exhibiting \emph{reversibility}
that allows invariant distributions to be computed explicitly as product-geometric distributions. We discuss invariant distributions and the generic absence of reversibility in the elastic model in \S\ref{sec:invariant_distributions} below. Rather than describe in detail the queueing model that corresponds to the exclusion systems (see~\cite[\S3]{mmpw} for that), in~\S\ref{sec:queues} we instead describe a similar correspondence for the elastic model. We refer to~\cite{mpw-ptrf}, and references therein, for exclusion interaction among semi-infinite systems of particles.

\subsection{Queueing networks with resource redeployment, random walks in orthants}
\label{sec:queues}

Once more, take the elastic process described in \S\ref{sec:definitions} in the case when all masses are equal, i.e., $m_i \equiv m \in (0,\infty)$ for all $i \in [N]$.
The Markov process $\eta = (\eta(t))_{t \in \RP}$ of inter-particle distances defined by~\eqref{eq:eta-def} has an interpretation as the process of queue-lengths in a network of $N-1$ queues in series, with a specific form of pooling of server resource. 

In this interpretation, $\eta_i(t) \in \ZP$ counts the number of customers in queue~$i \in [N-1]$ at time $t \in \RP$.
Customers enter and leave the network only at the two extremal queues, $1$ and $N-1$. Arrivals are according to a Poisson arrival stream with rate $a_1$ at queue~$1$, and an arrival stream of rate $b_{N}$ at queue~$N-1$. If queue $i \in [N-1]$ is non-empty, then the $i$th server serves the customer at the head of queue~$i$ at rate $b_i + a_{i+1}$. When a customer is served, it is routed to queue $i-1$ with probability $\frac{b_{i}}{b_i+a_{i+1}}$ or to queue $i+1$ with probability $\frac{a_{i+1}}{b_i+a_{i+1}}$; being routed to queue~$0$ or~$N$ means that the customer leaves the system. 

The above part of the dynamics is identical to the Jackson network that corresponds to the exclusion process: the difference for the elastic model will be that when a server's queue is empty, the server will \emph{redeploy} resource to serve the nearest occupied queues in each direction. Precisely, if queue $i \in [N-1]$ is empty, then server~$i$ lends service rate~$b_i$ to the nearest occupied queue to the right, if there is one, or to augment
the arrival stream at queue~$N-1$ if there are no occupied queues to the right, and lends
service rate~$a_{i+1}$ to the nearest occupied queue to the left, or to the arrival stream at queue~$1$ if there are no occupied queues to the left. The interpretation of the constant activity rate in~\eqref{eq:time-change} in this setting is that the total service effort in the queueing network is independent of the current configuration, unlike in the Jackson network. (Here `service' includes the server recruiting extra arrivals to the extremal queues.)  

The process $\eta$ can also be viewed as a continuous-time, reflected random walk on the orthant~$\ZP^{N-1}$, with nearest-neighbour jumps. 
For $N \in \{2,3\}$, there are exhaustive criteria for classifying stability of such
 walks (see e.g.~\cite{fim,fmm}\footnote{These works are written for discrete-time random walks, but they immediately apply to this case also, since the rates normalize to probabilities with the same normalizing constant from~\eqref{eq:time-change}.}) in terms of first and second moments
of increments. For $N \in \{4,5\}$ the generic classification  requires precise
 knowledge of quantities which are hard to compute, namely stationary distributions for
 lower-dimensional projections~\cite{fmm,ignatyuk}. For $N \geq 6$,  the generic case is intractable~\cite{gamarnik}. The family of random walks here, with the 
 explicit stability criterion in Corollary~\ref{cor:stability}, gives an example where there is   a complete
 stability description for any $N$; 
 the model of~\cite{mmpw} provides another class of examples.

\subsection{Atlas model of rank-dependent interacting diffusions}
\label{sec:atlas}

Continuum (diffusion) models, in which $N$ particles perform mutually-reflecting 
Brownian motions with drifts and diffusion coefficients determined by their ranks, 
have been extensively studied in recent years, and include the \emph{Atlas model} and its relatives:
see e.g.~\cite{bb,bfk,ipbkf,pp,kps,sarantsev,shiga} and references therein. 

In one of the most general settings for finite systems, the system of $N$~diffusing particles 
is described by a system of stochastic differential equations (SDEs) with reflections.
If $x_1 (t) \leq x_2 (t) \leq \cdots \leq x_N (t)$ denote the ordered positions on $\R$ of the $N$ particles at time~$t \in \RP$, 
and  parameters $q_j^-, q_j^+ \in (0,1)$ satisfy $q^+_{j+1} +q_j^- =1$ for all $j \in [N-1]$,
consider the dynamics satisfying
\begin{equation}
\label{eq:SDE-ordered-asymmetric}
\ud x_i (t) = \mu_{i} \ud t + \sigma_{i} \ud W_i (t) + q_i^+ \ud L^{i} (t) - q_i^- \ud L^{{i+1}} (t) 
, \text{ for every } i \in [N],
\end{equation}
where  $W_1, \ldots, W_{N}$ are independent standard Brownian motions,  
and $L^{i}(t)$ is the local time at~$0$ of $x_i-x_{i-1}$ up to time~$t$, with the convention $L^1 \equiv L^{N+1} \equiv 0$.
 The local-time terms in~\eqref{eq:SDE-ordered-asymmetric}
maintain the order of the~$x_i$, while each particle~$x_i$ has its own intrinsic drift ($\mu_i$) and diffusion ($\sigma_i$) coefficients. The coefficients $q_j^-, q_j^+$ multiplying the local time terms
control the relative impact of collisions on each of the two particles involved, and play 
a comparable  role to the \emph{masses} in the model of~\S\ref{sec:definitions}, as we now explain; see also the discussion in~\S3 of~\cite{kps}.

We can identify `masses' in the system~\eqref{eq:SDE-ordered-asymmetric} by defining, for $k \in [N]$,
\[ m_k := \prod_{i \in [k]} \left( \frac{1-q_i^+}{q_i^+} \right) . \]
Then some algebra shows that, for $k \in [N]$.
\[ \frac{m_{k-1}}{m_k+m_{k-1}} = q_k^+, \text{ and } \frac{m_{k+1}}{m_k + m_{k+1}} = 1 -q_{k+1}^+ = q_k^- ,\]
where we set $m_0 = m_{N+1} \equiv 0$. 
In other words, we can re-write~\eqref{eq:SDE-ordered-asymmetric} as, for every $ i \in [N]$, 
\begin{equation}
\label{eq:SDE-ordered-asymmetric-masses}
\ud x_i (t) = \mu_{i} \ud t + \sigma_{i} \ud W_i (t) + \frac{m_{i-1}}{m_i + m_{i-1}}  \ud L^{i} (t) - \frac{m_{i+1}}{m_i+m_{i+1}} \ud L^{{i+1}} (t)  . \end{equation}
The continuum system~\eqref{eq:SDE-ordered-asymmetric} is \emph{elastic} in the sense that the total flow of mass is constant. Formally, define $M: = \sum_{k \in [N]} m_k$ and 
the centre of mass $g(t) := M^{-1} \sum_{i \in [N]} m_i x_i (t)$. By summing~\eqref{eq:SDE-ordered-asymmetric-masses} and
using the conditions on the   $q_j^-, q_j^+$, we deduce that $g$ has dynamics
\[
\ud g (t) = \overline{\mu} \ud t + \overline{\sigma} \ud \overline{W} (t) , \]
where $\overline{W}$ is standard Brownian motion, and the drift and diffusion coefficients
satisfy
\[ \overline{\mu} = \frac{1}{M} \sum_{i \in [N]} m_i \mu_i, \text{ and } \overline{\sigma}^2 = \frac{N}{M^2} \sum_{i \in [N]} m_i^2 \sigma_i^2 .\]
(Compare Theorem~\ref{thm:mass-speed} below for the elastic particle system.)

The elastic continuum model just described serves as a scaling limit of 
a class of particle system models, under certain near-critical parameter scaling, that includes both
the elastic model of~\S\ref{sec:definitions} and the `inelastic' simple exclusion process of~\S\ref{sec:exclusion}: see~\cite[\S 3]{kps} and~\cite[\S 6.3]{mmpw}
for more detail. The equivalent scaling limit when the discrete-space model is viewed as a queueing network (as in~\S\ref{sec:queues}) is a \emph{heavy traffic} limit in which every queue is asymptotically critically loaded, and the scaling limit is a reflected diffusion on an orthant; see e.g.~\cite{reiman,gz}.

As far as the authors are aware, long-time stability of the general-parameter continuum system defined by~\eqref{eq:SDE-ordered-asymmetric} above has been studied only as far as classifying whether or not the whole system is stable (e.g.~\cite{pp,sarantsev,ipbkf,kps}), rather than identifying a full cloud decomposition as we do in Theorem~\ref{thm:cloud_decomposition}; in the continuum setting there are additional complexities as
one must consider the analogue of the `corner trapping' phenomenon for reflected diffusion in an orthant (see~\cite{varadhan-williams}).

\subsection{Further literature on elastic collisions}
\label{sec:literature}

Physical systems of particles
exhibit (approximately) elastic 
collisions, in which kinetic energy and 
momentum are both (approximately) conserved, if there are no external forces and internal energy states of particles can be (approximately) ignored. A number of probabilistic models for such physical systems have been studied in the literature, 
including  Kac's uniformly mixing model for kinetic theory of gases~\cite{kac} (which neglects space),
and the work of Jepsen, Harris and Spitzer~(see \cite{jepsen,harris65,spitzer69}) on  elastic models of colliding particles on the line with deterministic or Brownian free-space trajectories. To give a flavour of the sorts of phenomena observed for the spatial models, suppose that 
at time $0$ a system of particles with equal masses  is given  according to a stationary Poisson process on $\R$, with one much more massive tagged particle at the origin;   independently of the positions, the atoms are given i.i.d.\  velocities, and undergo elastic collisions when they meet. If one observes the  tagged particle in an infinite system started from a given density profile, possibly in the presence of a hard barrier, there are results that establish convergence to a Brownian, Ornstein--Uhlenbeck, or other Gaussian process, and also results when the tagged particle is much less  massive than the rest: see e.g.~\cite{gissel,bold86,braglia,ciesielski,dgl,lsy}. An informative overview of this and other literature is given in the introduction to~\cite{pp}.

\section{Proofs: Stability}
\label{sec:proofs-stability}

\subsection{Global speed and local stability}
\label{sec:speed}

In this section we state two key  ingredients to our proofs.
Theorem~\ref{thm:mass-speed} 
states that the drift of the centre of mass of the system is independent
of the current configuration, which  is a distinguishing feature of the elastic collision mechanism. 
Proposition~\ref{prop:partial-stability} gives a criterion for identifying (not necessarily maximal) stable subsystems of particles. These two results will combine to give Theorem~\ref{thm:stability_condition} (the short deduction appears after their statement),
and are also 
 crucial in the
proof of Theorem~\ref{thm:cloud_decomposition}.
The proofs of both
Theorem~\ref{thm:mass-speed} and Proposition~\ref{prop:partial-stability}
are given in \S\ref{sec:lyapunov} below.

Recall the
definition of the centre of mass process $G$ from~\eqref{eq:centre-of-mass}.
The next result shows that $G$ travels with speed $-U_N / M_N$.

\begin{theorem}[Centre of mass]
\label{thm:mass-speed}
For the process $G$ as defined at~\eqref{eq:centre-of-mass}, it holds that
\begin{equation}
    \label{eq:com-increment}
  \Exp [ G(t+h) - G(t) \mid X(t) = x ] = -h U_N/M_N, \text{ for all $t, h \in \RP$ and all } x \in \bbX_N. \end{equation}
Moreover, $\lim_{t \to \infty} t^{-1} G(t) = -U_N/M_N$, a.s.
\end{theorem}

The next important ingredient is a \emph{local stability}
result, presented in Proposition~\ref{prop:partial-stability} below. 
This result will yield the `stability' part of Theorem~\ref{thm:stability_condition}
directly. Moreover, 
the result says that if a contiguous subset of particles in the system
would,
considered as an isolated system,
satisfy the stability criterion from Theorem~\ref{thm:stability_condition},
then that subset is stable also as a subsystem of the full system.
The intuition
is that the presence of additional particles to the left or right can only stabilize the subsystem further, by preventing the extreme particles from diverging, and/or by contributing stabilizing inwards momentum. 

For $\ell , r \in [N]$
with $\ell \leq r$, we write
$[ \ell ; r ] := \{ x \in [N] : \ell \leq x \leq r \}$, a \emph{discrete interval}.
Note that $[1 ; r ] = [r]$ in our previous notation. 
Recall the definition of $\Delta_C$ from~\eqref{eq:radius-def}.

\begin{proposition}[Local stability]
    \label{prop:partial-stability}
    Suppose that $\ell, r \in [N]$, $\ell < r$, 
    and it holds that
    \begin{equation}
\label{eq:local-stability}
 \frac{U_{\ell,k}}{M_{\ell,k}} < \frac{U_{\ell,r}}{M_{\ell,r}} \text{ for all } k \in [ \ell ; r-1]. 
     \end{equation}
    Then there exist constants $C \in \RP$ and $\delta > 0$ (depending on the $a_i, b_i$, and $m_i$)
    such that 
    \begin{equation}
        \label{eq:exponential-bound}
\sup_{t \in \RP} \Pr \left[ \Delta_{[\ell+1;r]} (t) \geq s \right] \leq C \left[ 1 + \re^{C \Delta_{[\ell+1;r]} (0) } \right] \re^{-\delta s} , \text{ for all }  s \in \RP,     \end{equation} 
and, moreover,
    \begin{equation}
    \label{eq:log-bound}
    \limsup_{t \to \infty} \frac{\Delta_{[\ell+1;r]} (t)}{\log t} < \infty, \as \end{equation}
\end{proposition}
\begin{remark}
The statement is trivial when $\ell +1=r$, since it refers to a single particle. 
    Note that condition~\eqref{eq:local-stability} involving the functions $U_k, M_k$ over values $k \in [\ell;r]$ translates to stability of particles with labels $[\ell +1;k]$ at~\eqref{eq:exponential-bound}.
    This is similar to the stability criterion in Theorem~\ref{thm:cloud_decomposition}, reflecting that there is no particle with label~$0$.
\end{remark}  

\begin{proof}[Proof of Theorem~\ref{thm:stability_condition}]
The statement of Proposition~\ref{prop:partial-stability} with $\ell=0$ and $r =N$ shows that the condition~\eqref{eq:stability_condition} is sufficient for stability, in the sense expressed in part~\ref{thm:stability_condition-b} of the theorem. Moreover, since $\max_{i \in [N]} | X_i (t) - G(t) | \leq \Delta_{[N]} (t)$ for all $t \in \RP$, it follows also that $\limsup_{t \to \infty} t^{-1} | X_i (t) - G(t) | =0$, a.s., for every $i \in [N]$. Together with Theorem~\ref{thm:mass-speed}, this yields part~\ref{thm:stability_condition-a} of the theorem. 

The final statement of the theorem, that~\eqref{eq:stability_condition} is necessary for part~\ref{thm:stability_condition-b},  
is most easily deduced from Theorem~\ref{thm:cloud_decomposition}\ref{thm:cloud_decomposition-c} (the proof of that result comes only in~\S\ref{sec:proofs-instability}, but nowhere uses this final part of Theorem~\ref{thm:stability_condition}). Indeed, if~\eqref{eq:stability_condition} fails then there is at least one $j \in [\nu-1]$ for which $v_j \geq v_{j+1}$. If $v_j > v_{j+1}$, then Theorem~\ref{thm:cloud_decomposition}\ref{thm:cloud_decomposition-c} shows that $\liminf_{t \to \infty} t^{-1} \Delta_{[N]} (t) \geq v_j - v_{j+1} >0$, contradicting part~\ref{thm:stability_condition-b}. If $v_j = v_{j+1}$,
 then Theorem~\ref{thm:cloud_decomposition}\ref{thm:cloud_decomposition-c} shows that, for some $\eps>0$, $\limsup_{t \to \infty} t^{-\eps} \Delta_{[N]} (t) \geq 1$, a.s., say,
 which also contradicts  part~\ref{thm:stability_condition-b}.
\end{proof}

Before moving on to the proofs of Theorem~\ref{thm:mass-speed} and Proposition~\ref{prop:partial-stability}, we introduce some notation that will allow
us to work with increments of functionals of  particle system configurations. Define 
\begin{equation}
    \label{eq:tau-def}
\tau := \inf \{ t \in \RP : X(t) \neq X(0)\}, 
\end{equation}
the first jump time of the Markov chain. Write $\Pr_x, \Exp_x$ for probabilities
and expectations in the case where the initial configuration is~$X(0) = x \in \bbX_N$. 
The key step in the proof of Proposition~\ref{prop:partial-stability}
is the following Lyapunov function estimate. Recall the definitions of the inter-particle distance function $D : \bbX_N \to \ZP^{N-1}$ from~\eqref{eq:D-def} and $\eta (t) = D(X(t))$ defined at~\eqref{eq:eta-def}. Observe that $\tau$ defined by~\eqref{eq:tau-def} is also the first jump time of the Markov chain~$\eta$.

\begin{lemma}
    \label{lem:lyapunov}
    Suppose that $\ell, r \in [N]$, $\ell +1 < r$, and
        that~\eqref{eq:local-stability} holds. 
    Then there exist a function $F : \ZP^{N-1} \to \RP$ and constants $B, R \in \RP$ and $\eps >0$ such that 
    \begin{align}
        \label{eq:F-integrable}
       \IE_x    F(\eta(t) )  & < \infty, \text{ for all } x \in \bbX_N \text{ and all } t \in \RP;\\
        \label{eq:F-bound}
       \IP_x  [ | F(\eta(\tau)) - F(\eta(0)) | \leq B  ] & = 1, \text{ for all } x \in \bbX_N; \\
\label{eq:foster}
\IE_x [ F ( \eta(\tau) )  - F( \eta(0) ) ]
& \leq - \eps,
 \text{ for all } x \in \bbX_N \text{ with } \| D(x) \| > R.
\end{align}
    \end{lemma}

The function~$F$ will be a (weighted) norm describing the displacement 
of the particles with labels in $[\ell+1;r]$ relative to their centre of mass. 
In \S\ref{sec:lyapunov} we give the formal definition of~$F$, which requires introducing notation
for a suitable weighted norm and inner product, and then prove Lemma~\ref{lem:lyapunov} and deduce Proposition~\ref{prop:partial-stability}. The Markov-chain computations
required to establish Lemma~\ref{lem:lyapunov} also allow us to prove Theorem~\ref{thm:mass-speed}, so that proof is also in~\S\ref{sec:lyapunov}.

\subsection{Lyapunov function computations}
\label{sec:lyapunov}

 For $x,y\in \R^N$ and $\ell, r \in [N]$ with $\ell \leq r$, 
consider 
the weighted 
positive semi-definite form
\[ \langle x,y \rangle_{\ell,r} := \sum_{i \in [\ell;r]} m_i x_i y_i \]
and the associated seminorm $\|x\|_{\ell,r} := \sqrt{\langle x,x \rangle_{\ell,r}}$.
In the special case where $\ell =1$ and $r=N$, we write $\langle x,y \rangle_{N} := \langle x,y \rangle_{1,N}$
and $\|x\|_{N} := \|x\|_{1,N}$.

Let $\ones$ be the vector in $\R^N$ with all components equal to~$1$.
Observe that $ \langle \ones ,\ones \rangle_{\ell+1,r} = M_r - M_\ell = M_{\ell,r}$
in the notation at~\eqref{eq:U-M-diff-def}. In particular, $\langle \ones, \ones \rangle_N =
\langle \ones, \ones \rangle_{1,N}  = M_N$, and the centre of mass 
of a configuration $x\in\bbX_N$ is 
\begin{equation}
    \label{eq:centre-of-mass-config}
    g(x) := M_N^{-1} \langle x,\ones \rangle_N = \frac{1}{M_N} \sum_{i \in [N]} m_i x_i .
\end{equation}

For a configuration $x\in\bbX_N$,
let~$R_x$ be the total jump rate at~$x$,
i.e.,
\begin{equation}
\label{eq:R-def}
R_x := \lim_{h \to 0} h^{-1} \Pr_x ( X(h) \neq x ) = \sum_{i \in [N]} \bigl( A_i (x) + B_i (x) \bigr) , \end{equation}
where $A_i, B_i$ are defined at~\eqref{eq:rates}. 
It is not hard to check (see formula~\eqref{eq:expression_R_x} below) that
\begin{equation}
    \label{eq:lower-bound}
    \inf_{x \in \bbX_N} R_x > 0 \text{ provided } \sum_{i \in [N]} ( a_i + b_i) > 0.
\end{equation}
We will tacitly assume that $\sum_{i \in [N]} ( a_i + b_i) > 0$ for the rest of this section;
the case $\sum_{i \in [N]} ( a_i + b_i) = 0$ being trivial (but formally covered by the
results of~\S\ref{sec:definitions}). 
When a jump occurs from~$x$, 
the Markov chain transitions to some $x \pm e_k$, where $e_k \in \bbX_N$ is the unit vector with~1 in coordinate~$k$. The associated transition probabilities are, for every $k \in [N]$,
\begin{equation}
    \label{eq:jump-chain}
    \Pr_x [ X(\tau) = x - e_k ] = \frac{A_k (x)}{R_x}, \text{ and }
    \Pr_x [ X(\tau) = x + e_k ] = \frac{B_k (x)}{R_x},
\end{equation}
where $A_k, B_k$ are defined at~\eqref{eq:rates} and $R_x$ is given by~\eqref{eq:R-def}.
In what follows we write 
\begin{equation}
    \label{eq:e-def}
    \ee := X(\tau) -X(0)
\end{equation}
for the first increment, so that $\Pr_x ( \ee  = e)$
describe the transition probabilities for the (discrete-time) jump chain
as given by~\eqref{eq:jump-chain}. In particular, for $H : \bbX_N \to \RP$,
\begin{equation}
    \label{eq:H-increment}
    \Exp_x [ H ( X(\tau)) ] = \Exp_x H (x + \ee ), \text{ for all } x \in \bbX_N.
\end{equation}
Given $x \in \bbX_N$ and $k \in [N]$, we define $k_-, k_+ \in [N]$ by
\begin{equation}
    \label{eq:k-plus-minus}
k_- := k_-(x) :=\min \{ i\in[N] : x_i = x_k \}, ~ k_+ := k_+(x) := \max \{ i\in[N] : x_i = x_k \},    
\end{equation}
the minimal and maximal labels of particles sharing the same site as~$k$ in configuration~$x$.
We also define an `off-lattice' extension of $\bbX_N$ defined at~\eqref{eq:configuration-space}
by
\begin{equation}
\label{eq:configuration-space-shifted}
\bbX'_N := \bigl\{ (x_i )_{i \in [N]} \in \R^N : x_1 \leq x_2 \cdots \leq x_N \bigr\}.
\end{equation}
Write $a := (a_i)_{i \in [N]}$ and $b:= (b_i)_{i \in [N]}$ for the jump rates, and
recall from~\eqref{eq:intrinsic-speed} that $u=b-a$.
The following lemma enables us to work with expected functional increments;
for example, the hypotheses on $x,y$ in the statement are satisfied whenever $y_k = f (x_k)$ for all $k \in [N]$ and some non-decreasing $f : \Z \to \R$,
and that case motivates why we defined the off-lattice space $\bbX'_N$ at~\eqref{eq:configuration-space-shifted}.

\begin{lemma}
\label{lem:drift_y}
 Let $x\in\bbX_N$ and assume that $y\in\bbX'_N$
is such that $y_k=y_{k+1}$ for all~$k \in [N]$
such that $x_k=x_{k+1}$. Then 
for $\ee$ as defined at~\eqref{eq:e-def}, it holds that
\begin{equation}
\label{eq:drift_y}
\IE_x \langle y,\ee \rangle_{N} 
= R_x^{-1} \langle  y , u \rangle_N .
\end{equation}
More generally, for $\ell, r \in [N]$, 
with $\ell \leq r$ and 
 $\ell_\pm, r_\pm$ as defined in~\eqref{eq:k-plus-minus},
\begin{align}
\label{eq:drift_y-partial}
 R_x \IE_x \langle y,\ee \rangle_{\ell,r} 
  & = \langle y ,u \rangle_{\ell , r} 
  + \2 { \{ \ell > \ell_- \}} \langle y, b \rangle_{\ell_-, \ell - 1} 
  + \2 { \{ \ell > \ell_- \}} \langle y, a \rangle_{\ell_-, \ell_+}  - \2 {\{ r < r_+ \}} \langle y , a \rangle_{r+1, r_+} \nonumber\\
& {} \qquad{}  
  - \2 {\{ r < r_+ \}} \langle y ,b \rangle_{r_- ,r_+}.
\end{align}
In particular, if $\min_{k \in [N]} y_k \geq 0$, we have the inequalities
\begin{align}
\label{eq:drift-y-left}
\IE_x \langle y,\ee \rangle_{1,r} 
& \leq R_x^{-1}\langle y , u \rangle_{1, r}, \\
\label{eq:drift-y-right}
\IE_x \langle y,\ee \rangle_{\ell,N} 
& \geq R_x^{-1}\langle y ,u \rangle_{\ell, N}.
\end{align}
\end{lemma}

To prepare for the proof, we introduce some convenient terminology for describing configurations.
Take $s \in [N]$ and
$h_1, \ldots, h_s \in \N$ such that $\sum_{j \in [s]} h_j = N$.
For $k \in [s]$, write $H_k := \sum_{j=1}^k h_j$, so that $H_0 =0$ and $H_s = N$.
We say that $x = (x_1, \ldots, x_N) \in \bbX_N$
is a configuration with $s$~\emph{stacks} of \emph{heights} $h:= (h_1, \ldots, h_s)$, 
and write $x \in \bbX^{s,h}_N$, if
\begin{equation}
    \label{eq:stacks}
x_{1} < x_{H_1+1} < \cdots < x_{H_{s-1}+1}, \text{ and, for all $j \in [s]$ and $i \in [h_j]$, }
x_{H_{j-1} +i} = x_{H_{j-1} +1}.
\end{equation}
Figure~\ref{fig:elastic} shows a configuration $x \in \bbX^{3,h}_6$
of $6$~particles in~$s=3$ stacks of heights $h = (1,2,3)$.

\begin{proof}[Proof of Lemma~\ref{lem:drift_y}]
 Suppose that $x \in \bbX_N^{s,h}$ for $s \in [N]$
and $h = (h_1,\ldots,h_s)$, as at~\eqref{eq:stacks}.
First observe, by~\eqref{eq:R-def} and~\eqref{eq:rates}, we may write
\begin{equation}
\label{eq:expression_R_x}
 R_x = \sum_{j=1}^s 
  \sum_{i=1}^{h_j}
  \Big( a_{H_{j-1} +i} \frac{m_{H_{j-1}+i}}{m_{H_{j-1}+1}}
   + b_{H_{j-1}+i}\frac{m_{H_{j-1}+i}}{m_{H_j}}\Big),
\end{equation}
and that, by~\eqref{eq:jump-chain}, the increment~$\ee$ is equal to~$-e_{H_{j-1}+1}$
and to~$e_{H_j}$
with probabilities, respectively,  
$R_x^{-1}\sum_{i=1}^{h_j}
  a_{H_{j-1}+i}\frac{m_{H_{j-1}+i}}{m_{H_{j-1}+1}}$
and 
$R_x^{-1}\sum_{i=1}^{h_j}
  b_{H_{j-1}+i}\frac{m_{H_{j-1}+i}}{m_{H_j}}$.
 Therefore, 
 \begin{align*}
\label{eq:drift_y-partial}
 R_x \IE_x \langle y,\ee \rangle_{\ell,r} 
   &  = \langle y ,u \rangle_{\ell_+ + 1, r_- - 1} 
   + \langle y , b - a\2 { \{ \ell = \ell_- \} } \rangle_{\ell_-, \ell_+}  + \langle y , b \2 { \{ r = r_+ \}} - a \rangle_{r_-, r_+} .
\end{align*}
This yields~\eqref{eq:drift_y-partial}. Taking $\ell = 1$ (which certainly has $\ell = \ell_-$)
in~\eqref{eq:drift_y-partial}
we get
\begin{align*}
R_x \IE_x \langle y,\ee \rangle_{\ell,r} 
&   = \langle y , u \rangle_{\ell, r}
   - \2 {\{ r < r_+ \}} \langle y ,b \rangle_{r_-, r_+}
   - \2 {\{ r < r_+ \}} \langle y, a \rangle_{r +1, r_+}   \leq \langle y, u \rangle_{\ell_- , r_+}, 
\end{align*}
provided all~$y_k \geq 0$, 
as claimed in~\eqref{eq:drift-y-left}. Similarly, taking $r = r_+ = N$ in~\eqref{eq:drift_y-partial} yields~\eqref{eq:drift-y-right}. Taking both $\ell =1$ and $r =N$ in~\eqref{eq:drift_y-partial}, we verify~\eqref{eq:drift_y}.
\end{proof}

Define $\tau_0 := 0$ and $\tau_1 := \tau$ as at~\eqref{eq:tau-def}, and then, iteratively, define \begin{equation}
    \label{eq:tau-sequence}
\tau_{n+1} := \inf \{ t \geq \tau_n : X(t) \neq X(\tau_n) \}, \text{ for every } n \in \ZP, 
\end{equation}
so that $0 = \tau_0 < \tau_1 < \tau_2 < \cdots$
are the (a.s.~finite) jump times of~$X$. 
Let $\cF_t := \sigma ( X(s) : 0 \leq s \leq t)$, the $\sigma$-algebra
generated by $X$~up to time~$t \in \RP$.
By~\eqref{eq:R-def}, we have that, given~$\cF_{\tau_n}$,
the holding time~$\tau_{n+1} - \tau_n$ is exponentially distributed with parameter~$R_{X(\tau_n)}$,
which by~\eqref{eq:R-def}--\eqref{eq:lower-bound}
satisfies $\delta < R_x < \delta^{-1}$ for every $x \in \bbX_N$ and some constant $\delta \in (0,1)$ depending only on the $a_i, b_i$, and $m_i$. 
Observe that since $\eta = D(X)$ jumps if and only if $X$ jumps, the $\tau_n$ given by~\eqref{eq:tau-sequence} are also the jump times of~$\eta$, a fact that we will use in the proof of Proposition~\ref{prop:partial-stability} below.
First we can complete the proof of Theorem~\ref{thm:mass-speed}.

\begin{proof}[Proof of Theorem~\ref{thm:mass-speed}]
Comparison of notation~\eqref{eq:centre-of-mass} for $G$ and~\eqref{eq:centre-of-mass-config} for $g$ reveals that $G(t) = g (X(t))$. Then by~\eqref{eq:H-increment} and an application of
Lemma~\ref{lem:drift_y}
with $y=\ones$ shows, using the formula~\eqref{eq:centre-of-mass-config}, 
\begin{align}
\label{eq:com-drift}
    \lim_{h \to 0} h^{-1} \Exp [ G (t+h) - G(t) \mid X(t) = x ] & = R_x  \Exp_x [ g (X(\tau)) - g(x) ] \nonumber\\
    & = R_x \bigl( \Exp_x g (x + \ee) - g(x) \bigr) \nonumber\\
    & = R_x M_N^{-1} \Exp_x \langle  \ee, \ones \rangle_N  \nonumber\\
    & = - \frac{U_N}{M_N},  \text{ for all } x \in \bbX_N.
\end{align}
Since the final expression in~\eqref{eq:com-drift} is independent of~$x$, we verify~\eqref{eq:com-increment}. Moreover, considering the associated jump chain, for every $k \in \ZP$, 
\begin{equation}
\label{eq:com-jump} \Exp [ g (X(\tau_{k+1}) ) - g (X(\tau_k)) \mid \cF_{\tau_k} ] = - \frac{U_N}{M_N R_{X(\tau_k)}} , \as \end{equation}
As argued below~\eqref{eq:tau-sequence},
the holding times~$\tau_{k+1} - \tau_k$ are exponentially distributed
with uniformly bounded rates, and $\Exp [ \tau_{k+1} - \tau_k \mid \cF_{\tau_k} ] = 1/R_{X(\tau_k)}$.
Set $A_0 :=0$ and $A_k := \sum_{\ell=0}^{k-1} 1/R_{X(\tau_\ell)}$ for $k \in \N$. 
Then $\tau_k - A_k$ is a martingale, and we can write $\tau_k - A_k = \sum_{\ell=0}^{k-1} D_\ell$
where $D_\ell := \tau_{\ell+1}-\tau_\ell - R^{-1}_{X(\tau_\ell)}$ are martingale differences. By orthogonality of martingale differences,
$\Exp [ (\tau_k - A_k)^2 ] = \sum_{\ell=0}^{k-1} \Exp [ D_\ell^2 ] \leq C k$,
for some $C< \infty$ depending only on the $a_i$, $b_i$, and $m_i$. Then an application of Doob's inequality to the non-negative submartingale $|\tau_k - A_k|$ shows that
\[ \Pr \Bigl[ \max_{0 \leq \ell \leq k} \bigl| \tau_\ell - A_\ell \bigr| > k^{2/3} \Bigr] \leq \frac{C k}{k^{4/3}} = O (k^{-1/3} ) .\]
The Borel--Cantelli lemma shows that a.s., for all but finitely many $n \in \ZP$,
 $\max_{0 \leq \ell \leq 2^n} \bigl| \tau_\ell - A_\ell \bigr| \leq (2^n)^{2/3}$. Every $k \in \N$ has $2^n \leq k < 2^{n+1}$ for some $n$; hence, a.s., for all but finitely many $k \in \ZP$,
 \[ \max_{0 \leq \ell \leq k} \bigl| \tau_\ell - A_\ell \bigr| \leq  \max_{0 \leq \ell \leq 2^{n+1}} \bigl| \tau_\ell - A_\ell \bigr| \leq ( 2^{n+1} )^{2/3} < 2 k^{2/3} .\]
 It follows that
\begin{equation}
\label{eq:slln-1}  \lim_{k \to \infty} \frac{1}{k} \left[ \tau_k - A_k    \right] = 0 , \as,  \end{equation}
Since, a.s., $\delta < k^{-1} A_k < \delta^{-1}$ for all $k \in \N$,
it follows from~\eqref{eq:slln-1} that $\tau_k/A_k \to 1$, a.s., as $k \to \infty$.
Furthermore, there exists $B \in \RP$ such that $| g ( X(\tau_{k+1} ) ) - g (X ({\tau_k} )) | \leq B$ for all $k \in \ZP$, a.s.
 The standard Doob construction, with compensator given via~\eqref{eq:com-jump}, then shows that $M_k := g(X(\tau_k)) + (U_N/M_N) A_k$ defines a martingale
 with uniformly bounded increments, i.e., $|M_{k+1} - M_k | \leq C$, a.s., for some constant $C<\infty$. An application of the  Azuma--Hoeffding inequality (see e.g.~\cite[p.~46]{mpw-book}) yields
 $ \Pr [ | M_k - M_0 | \geq k^{2/3} ] \leq 2 \exp ( - k^{1/3}/(2C^2))$ for all $k \in \N$. Hence, by the Borel--Cantelli lemma, 
\begin{equation}
    \label{eq:slln-2}
     \lim_{k \to \infty} \frac{1}{k} \left[ g (X(\tau_k)) + (U_N/M_N) A_k    \right]  = 0, \as
\end{equation}
 Combining~\eqref{eq:slln-1} and~\eqref{eq:slln-2}, we get
$\lim_{k \to \infty} g (X(\tau_k))/ \tau_k =  - U_N/M_N$, a.s.
Now every $t \in \RP$ has $\tau_k \leq t < \tau_{k+1}$ for some $k = k_t$ (with $k_t \to \infty$ as $t \to \infty$) and $t^{-1} G(t) = \tau_k^{-1} G(\tau_k) \cdot (\tau_k /t)$, where
\[ \left| \frac{\tau_k}{t} -1 \right| \leq \frac{| \tau_{k+1} - \tau_k|}{\tau_k} .\]
Here, by standard properties of exponential random variables,   $| \tau_{k+1} - \tau_k | \leq C \log k$, a.s., for all but finitely many $k \in \ZP$, while we saw above that $\tau_k /k \to \infty$ a.s. Hence $\tau_k /t \to 1$ and we conclude that  $\lim_{t \to \infty} t^{-1} G(t) = -U_N/M_N$, a.s.
\end{proof}

Next we turn to the proof of Lemma~\ref{lem:lyapunov}.

\begin{proof}[Proof of Lemma~\ref{lem:lyapunov}]
Suppose $\ell+1 <r$.
For~$x\in\bbX_N$, 
  set $x_{[\ell;r]}= ( x^{(1)}_{[\ell;r]} , \ldots, x^{(N)}_{[\ell;r]}) \in \bbX_N$, where
\[  x^{(k)}_{[\ell;r]} =  x_k \2 { \{ k \in [\ell ; r] \} },\]
i.e.,  projection onto the subspace spanned by coordinates~$[\ell;k]$.
Define $\psi_{\ell,r} : \bbX_N \to \bbX'_N$ by
\begin{equation}
\label{eq:psi-def}
\psi_{\ell,r} (x) := x_{[\ell+1;r]} - M_{\ell,r}^{-1} \langle x,\ones\rangle_{\ell+1,r}  \ones_{[\ell+1;r]} , \text{ for } x \in\bbX_N,
\end{equation}
i.e., projection of~$x$ onto coordinates $[\ell+1;r]$, translated by its centre of mass.
Note that~$\psi_{\ell,r}$ is invariant under translation of~$x$, since if $y_i = x_i + \alpha$ for all $i \in [N]$ and some $\alpha \in \R$, $\langle y , \ones \rangle_{\ell,r} = \langle x , \ones \rangle_{\ell,r} + \alpha M_{\ell,r}$.

Since $\psi_{\ell,r} (x)$ depends only on the $[\ell+1;r]$ coordinates
of $x$, we could view $\psi_{\ell,r}$ as acting from $\bbX_{r-\ell}$ to $\bbX'_{r-\ell}$,
but we prefer to keep the full vectors (including all zeros) since coordinate labels correspond to specific masses. The Lyapunov function~$F$ that appears in~\eqref{eq:foster} will be derived from
\begin{equation}
    \label{eq:Psi-def}
    \Psi (x) := \| \psi_{\ell,r} (x) \|_{\ell+1,r}, \text{ for all } x \in \bbX_N
    ,
\end{equation}
where $\psi_{\ell,r}$ is given by~\eqref{eq:psi-def} (we suppress dependence on $\ell,r$ in the notation $\Psi$). 
Since $\Psi$ is invariant under translation of~$x$, $\Psi(x)$ can be written as a function of~$D(x)$
as given at~\eqref{eq:D-def}. To do so,   define $D^{-1} : \ZP^{N-1} \to \bbX_N$ by 
$D^{-1} (z_1, \ldots, z_{N-1} ) := ( 0 , z_1, z_1+z_2, \ldots, z_1 + \cdots + z_{N-1} )$, and then 
\begin{equation}
    \label{eq:F-def}
    F ( z ) := \Psi (D^{-1} (z) ) \text{ for all } z \in \ZP^{N-1}
\end{equation}
will be the~$F$ that appears in~\eqref{eq:foster}. We also note from~\eqref{eq:Psi-def} and~\eqref{eq:psi-def} that
\begin{align}
\label{eq:Psi-mass-variance}
    \Psi (x)^2 &=  \| \psi_{\ell,r} (x) \|_{\ell+1,r}^2 = M_{\ell,r} \sum_{i \in [\ell+1;r]} \frac{m_i}{M_{\ell,r}} \left[ x_i - \sum_{j \in [\ell+1;r]} \frac{m_j}{M_{\ell,r}} x_j \right]^2 ,
\end{align}
which exhibits~$\Psi^2$ as proportional to the variance of the mass distribution if particles $\ell+1, \ldots, r$ are placed at locations $x_{\ell+1}, \ldots, x_r$. It follows from~\eqref{eq:Psi-mass-variance} that
\begin{equation}
    \label{eq:Psi-upper-bound}
    \Psi (x)^2 \leq M_{\ell,r} \left| \max_{i \in [\ell+1;r]} x_i - \min_{i \in [\ell+1;r]} x_i \right|^2
    = M_{\ell,r} (x_r - x_{\ell+1} )^2 .
\end{equation}
In the other direction, since $a^2 +b^2 \geq (a+b)^2/2$, we get from~\eqref{eq:Psi-mass-variance} that
\begin{equation}
    \label{eq:Psi-lower-bound}
    \Psi (x)^2 \geq m_{\ell+1} (\bar x - x_{\ell+1})^2 + m_r ( x_r - \bar x)^2 
    \geq \frac{\min ( m_{\ell+1}, m_r )}{2} (x_r - x_{\ell+1} )^2 ,
\end{equation}
where we wrote $\bar x := M_{\ell,r}^{-1} \sum_{j \in [\ell+1;r]} m_j x_j$. One can interpret bounds~\eqref{eq:Psi-upper-bound}--\eqref{eq:Psi-lower-bound} as quantifying the fact that
$x_r - x_{\ell+1} = \max_{i,j \in [\ell+1;r]} | x_i - x_j|$
is equivalent to the norm $\| \psi_{\ell,r} (x) \|_{\ell+1,r}$.

We need to study the increments of~$\Psi$ applied to Markov process $X$.
Since $\langle \ones,\ones \rangle_{\ell+1,r} =M_{\ell,r}$, $\psi_{\ell,r}$ given by~\eqref{eq:psi-def} satisfies the orthogonality relation
\begin{equation}
    \label{eq:orthogonality}
    \langle \psi_{\ell,r} (x) , \ones_{[\ell+1;r]} \rangle_{\ell+1,r} = 0, \text{ for all } x \in \bbX_N .
\end{equation}
Let $e=\pm e_j$ be (plus or minus)
one of the coordinate
unit vectors.
We have $\psi_{\ell,r} (x+e) = \psi_{\ell,r} (x)$
whenever $j \notin [\ell+1;r]$. So fix $j \in [\ell+1;r]$. Then
we compute
\begin{align*}
\|\psi_{\ell,r} (x+e)\|^2_{\ell+1,r} &=
 \bigl\|x_{[\ell+1;r]}+e - M_{\ell,r}^{-1} \langle x+e,\ones \rangle_{\ell+1,r} \ones_{[\ell+1;r]} \bigr\|_{\ell+1,r}^2 \\
 &= \bigl\|  \psi_{\ell,r}(x)+e-M_{\ell,r}^{-1} \langle e,\ones \rangle_{\ell+1,r} \ones_{[\ell+1;r]} \bigr\|^2_{\ell+1,r} \\
 &= \|\psi_{\ell,r} (x)\|_{\ell+1,r}^2 + 2 \langle \psi_{\ell,r} (x),e \rangle_{\ell+1,r} + m_j -M_{\ell,r}^{-1}m_j^2,
\end{align*}
using the orthogonality~\eqref{eq:orthogonality} of $\psi_{\ell,r}(x)$ and~$\ones_{[\ell+1;r]}$.
With~$\ee$ being the random increment vector
at~$x\in\bbX_N$, it follows from the above that
\begin{equation*}
     \Psi (x+\ee)^2 -  \Psi (x)^2 
  = 2   \langle \psi_{\ell,r} (x), \ee \rangle_{\ell+1,r} + T_{\ell,r} (x, \ee),
\end{equation*}
where $| T_{\ell,r} (x,e) |$ is uniformly bounded for all $x \in \bbX_N$ and all unit vectors~$e$.
Moreover, $| \Psi (x + e) - \Psi (x) |$ is also  uniformly bounded for all $x \in \bbX_N$ and all unit vectors~$e$. It follows that there is a constant $B < \infty$ such that, a.s.,
\begin{align*}
 \Psi (x + \ee) - \Psi (x) & = \frac{\Psi (x+\ee)^2 - \Psi (x)^2}{\Psi (x+\ee)+\Psi(x)}  \leq  \frac{ \langle \psi_{\ell,r} (x), \ee \rangle_{\ell+1,r}}{\Psi(x)} + \frac{B}{\Psi(x)} ,
\end{align*}
for all $x \in \bbX_N$. In particular, for all $x \in \bbX_N$, 
\begin{equation}
    \label{eq:unsquare-norm}
  \IE_x \left[  \Psi (x+\ee) -  \Psi (x) \right] 
  \leq \frac{1}{\Psi(x)} \left[ \IE_x  \langle \psi_{\ell,r} (x), \ee \rangle_{\ell+1,r} + B \right]. 
\end{equation}
The remaining part of the proof is to obtain a (sufficiently negative) upper bound for the right-hand
expectation in~\eqref{eq:unsquare-norm}.
 To this end, define
\begin{equation}
    \label{eq:V-def}
    V_k := U_{\ell,k} - M_{\ell,k} \frac{U_{\ell,r}}{M_{\ell,r}}, \text{ for } k \in [\ell;r] .
\end{equation}
Then $V_\ell = V_r = 0$, and 
the condition~\eqref{eq:local-stability}
implies that $V_k <0$ for all $k \in [\ell+1 ; r-1]$ (and $U_{\ell,r} > 0$).
Hence, under the hypothesis~\eqref{eq:local-stability}, assuming $\ell+1<r$, 
\begin{equation}
    \label{eq:eps-def}
    d := - \max_{k \in [\ell+1; r-1]} V_k 
\end{equation}
satisfies $d >0$.
Moreover, since $U_{k} - U_{k-1} = m_k u_k$,
a straightforward calculation implies that
\begin{equation}
    \label{eq:V-diff}
 V_{k-1}- V_k
 = m_k \left[ u_k + \frac{U_{\ell,r}}{M_{\ell,r}} \right], \text{ for all } k \in [\ell+1; r].
\end{equation}
We will next apply Lemma~\ref{lem:drift_y} 
with $y = \psi_{\ell,r} (x)$. 
Note first that for $k \in [\ell_-, \ell_+]$, we have $x_k = x_\ell = \min_{i \in[\ell+1;r]} x_i$,
from which, using~\eqref{eq:psi-def}, it follows that $y_k \leq 0$ for all $k \in [\ell_-;\ell_+]$. Similarly,
$y_k \geq 0$ for all $k \in [r_-,r_+]$. Hence we deduce from~\eqref{eq:drift_y-partial} that $R_x \IE_x \langle \psi_{\ell,r} (x),\ee \rangle_{\ell+1,r} 
 \leq \langle y ,u \rangle_{\ell+1,r}$.
 Since $\sum_{k \in [\ell+1;r]} y_k m_k = 0$, by~\eqref{eq:orthogonality},
 it follows that
 \begin{align*}
R_x \IE_x \langle \psi_{\ell,r} (x),\ee \rangle_{\ell+1,r} 
& \leq \sum_{k \in [\ell+1;r]} y_k m_k \left[ u_k + \frac{U_{\ell,r}}{M_{\ell,r}} \right]
  = \sum_{k \in [\ell+1;r]} y_k \left[ V_{k-1} - V_k \right], 
\end{align*}
by~\eqref{eq:V-diff}. Then, by partial summation and the fact that $V_\ell = V_r =0$,
\begin{align*}
R_x \IE_x \langle \psi_{\ell,r} (x),\ee \rangle_{\ell+1,r} 
& \leq \sum_{k \in [\ell +1 ; r-1]} V_k ( y_{k+1} - y_k )
 \leq -  d \sum_{k \in [\ell +1; r-1]} ( y_{k+1} - y_k ), 
\end{align*}
using~\eqref{eq:eps-def} and the fact that $y_{k+1} - y_k = x_{k+1} - x_k \geq 0$ for $k \in [\ell +1; r-1]$. Hence
\begin{align*}
R_x \IE_x \langle \psi_{\ell,r} (x),\ee \rangle_{\ell+1,r} 
 & \leq -  d ( y_r - y_{\ell+1} ) = - d ( x_r - x_{\ell+1} ) .
\end{align*}
Using the inequality~\eqref{eq:Psi-upper-bound} we conclude that there exists $\eps >0$ 
 such that
\begin{equation}
\label{eq:square-norm}
  \IE_x \langle \psi_{\ell,r} (x),\ee \rangle_{[\ell+1;r]} 
  \leq -  \eps \Psi (x)  , \text{ for all } x \in \bbX_N.
\end{equation}
Combining~\eqref{eq:unsquare-norm} and~\eqref{eq:square-norm}, we see that, for some $R \in \RP$, 
\[ \IE_x ( \Psi (X(\tau) ) - \Psi (X(0) ) ) \leq - \eps, \text{ for all } x \in \bbX_{N} \text{ with } \| D(x) \| > R .\]
Since $\Psi$ is translation invariant, from here and~\eqref{eq:F-def} we deduce~\eqref{eq:foster}.
The fact that  $| \Psi (x + e) - \Psi (x) |$ is uniformly bounded implies~\eqref{eq:F-bound}, and~\eqref{eq:F-integrable} follows.
\end{proof}

To prove Proposition~\ref{prop:partial-stability}, we will show that  the Lyapunov function~$F$
from Lemma~\ref{lem:lyapunov}
satisfies a  version of a geometric drift condition (see e.g.~Definition 14.1.5 of~\cite{douc}).
The theory of \emph{geometric ergodicity}  for Markov chains (c.f.~Theorem 15.1.5 of~\cite{douc}) does  not apply directly for us,
as through~$F$ we are working with a \emph{projection} of a Markov process,
considering the relative stability of some (typically, not all) coordinates. Nevertheless,
we need less than the full power of geometric ergodicity;  the  Lyapunov function technique relies on martingale rather than Markov structure, and we can obtain
what we need from robust results for adapted processes from~\cite{hajek}. 

\begin{proof}[Proof of Proposition~\ref{prop:partial-stability}]
First note that if $\ell+1=r$ the result is trivial, since $\Delta_{[\ell+1;r]} (t) \equiv 0$ as defined at~\eqref{eq:radius-def}. Thus it suffices to suppose $\ell+1 < r$, so that we may apply Lemma~\ref{lem:lyapunov}.
The first part of the proof is to derive bounds on increments of $\eta$ over fixed time intervals,
 using the bounds over stochastic time intervals from Lemma~\ref{lem:lyapunov}. 
To this end, for $h \in (0,\infty)$
let $N_h$ be the number of times that $\eta$ jumps in time interval $[0,h]$.
Then $\IP_x ( N_h = 1) \geq p h$ for some $p>0$, all $h \in [0,h_0]$ for sufficiently small $h_0 \in \RP$, and all $x \in \bbX_N$, since~\eqref{eq:local-stability}
implies that at least one $a_i$ or $b_i$ is positive. On the other hand, there is some $C_1 <\infty$ (depending only on the $a_i$, $b_i$, and $m_i$) such that $N_h$ is stochastically dominated by a Poisson random variable with mean $C_1 h$, uniformly in $\eta(0)$. Note also, using~\eqref{eq:F-bound}, that
$| F(\eta(h)) - F(\eta(0)) | \leq B N_h$. 

Observe that, for $\delta \in \RP$,
\begin{align}
\label{eq:exp-F}
    \re^{\delta ( F(\eta(h)) - F (\eta(0)))}
    & \leq \1 {N_h =0 } + \re^{\delta ( F(\eta(h)) - F (\eta(0)))} \1 { N_h = 1} + \re^{\delta B N_h } \1 { N_h \geq 2} .
\end{align}
If  $Z$ is Poisson with mean $\lambda$, then a calculation shows that, for $\theta \in \R$ and $n \in \N$,
\begin{equation}
\label{eq:poisson-bound}
\Exp [ \re^{\theta Z} \1{ Z \geq n } ] \leq \lambda^n \exp ( \lambda (\re^\theta-1)). \end{equation}
By Taylor's theorem with Lagrange remainder, 
for every $\delta_0 \in (0,\infty)$, with $C_2 := \re^{\delta_0 B}$, 
for all $\delta \in [0,\delta_0]$ and all $y \in [-B,B]$,
\begin{equation}
\label{eq:exp-taylor}
\bigl| \re^{\delta y} - ( 1+\delta y) \bigr| \leq C_2 \delta^2 y^2 .\end{equation}
Taking expectations in~\eqref{eq:exp-F} and using~\eqref{eq:poisson-bound} and~\eqref{eq:exp-taylor}, we get, for all $\delta \in [0,\delta_0]$ and all $h \in [0,h_0]$,
\begin{align*}
    \Exp_x \re^{\delta ( F(\eta(h)) - F (\eta(0)))}
    & \leq 1 + \delta \Exp_x \left[ ( F(\eta(h)) - F (\eta(0))) \1 { N_h = 1} \right] \\
& {} \quad{}    + C_2 \delta^2 \Exp_x \left[ \big( F(\eta(h)) - F (\eta(0))\big)^2 \1 { N_h = 1} \right]
    + \Exp_x \bigl[ \re^{\delta B N_h } \1 { N_h \geq 2} \bigr] \\
    & \leq 1 + \delta \Pr_x ( N_h = 1) \Exp_x \left[ F (\eta (\tau) ) - F (\eta(0)) \right] \\
    & {} \quad {}     + C_2 \delta^2 B^2 \Pr_x ( N_h = 1)  + C_1^2 h^2 \exp ( C_1 h (\re^{\delta_0 B} -1)) .
\end{align*}
Now,  using~\eqref{eq:foster} (with $\eps$ defined as there), we get, for $\| D(x) \| > R$,
\begin{align*}
    \Exp_x \re^{\delta ( F(\eta(h)) - F (\eta(0)))}
    & \leq 1 - \delta p h \eps  + C_1 C_2 \delta^2 B^2 h  + C_1^2 h^2 \re^{ C_1 C_2 h} .
\end{align*}
Fix $\delta_0 \in (0,\infty)$. 
Choose $\delta = A h$ and $A$ large enough so that $A  p \eps > 3 C_1^2 \re^{C_1 C_2 h_0 }$. Then choose $h_1 \in (0, h_0)$ small enough so that $A h_1 < \delta_0$ and
$ p \eps  > 3 C_1 C_2 A h_1 B^2$. Choose $h_2 \in (0,h_1)$. With this choice of $\delta \in (0,\delta_0)$ it follows that  there exists $\rho <1$ such that, for all $h \in [h_2, h_1]$,
 \begin{align}
 \label{eq:hajek-check-1}
    \Exp_x \re^{\delta ( F(\eta(h)) - F (\eta(0)))}
    & \leq \rho, \text{ whenever }  \| D(x) \| > R.
\end{align}
On the other hand, using the bound $| F(\eta(h)) - F(\eta(0)) | \leq B N_h$ it is not hard to show that
\begin{equation}
\label{eq:hajek-check-2} \Exp_x \re^{\delta ( F(\eta(h)) - F (\eta(0)))}
     \leq D ,\end{equation}
    for some $D<\infty$ uniformly in $h \in [h_2,h_1]$.

We aim to apply Theorem~2.3 of~\cite{hajek}, which says the following. Suppose that $Y_0, Y_1, \ldots$ is a $\RP$-valued stochastic process adapted to a filtration $\cG_0, \cG_1, \ldots$. Suppose that there exist constants $\delta >0$, $\rho \in \RP$, $D < \infty$, such that, for all $k \in \ZP$,
\begin{align} \Exp [ \re^{\delta ( Y_{k+1} - Y_k)} \mid \cG_k ] &\leq \rho, \text{ on } \{ Y_k > a\}; \label{eq:hajek-1} \\
\Exp [ \re^{\delta ( Y_{k+1} - Y_k)} \mid \cG_k ] & \leq D , \text{ on } \{ Y_k \leq a \}.\label{eq:hajek-2} \end{align}
Then for all $k \in \ZP$ and all $r \geq 0$,
\[ \Pr [ Y_k \geq r ] \leq \rho^k \re^{\delta (Y_0 -r)} + \frac{1-\rho^k}{1-\rho} D \re^{\delta (a-r)} .\]
We apply this result with $Y_k = F (\eta (kh ))$ and $\cG_k = \cF_{\tau_k}$;
conditions~\eqref{eq:hajek-1}--\eqref{eq:hajek-2} are verified by~\eqref{eq:hajek-check-1}--\eqref{eq:hajek-check-2} above.
Given $t \in \RP$ and $h_1, h_2$,
set $k = \lfloor t / h_2 \rfloor$ so $t = h k$ with $h \geq h_2$ and $h \leq h_1$ if $t \geq (h_1 h_2)/ (h_1 - h_2) =: t_0$.
It follows  that there exist constants $C \in \RP$ and $\delta >0$,   such that 
\begin{equation}
    \label{eq:hajek-bound}
   \sup_{t \geq t_0} \IP_x [ F (\eta (t) ) \geq r   ]  \leq  C \left[ 1 + \re^{\delta F(D(x))  } \right]  \re^{-\delta r} , \text{ for all } r \in \RP.
\end{equation}
For $t \in [0,t_0]$ we use a bound similar to~\eqref{eq:hajek-check-2} and Markov's inequality to show that~\eqref{eq:hajek-bound} holds with the supremum over $t \geq t_0$ extended to all $t \geq 0$.
A consequence of the inequalities~\eqref{eq:Psi-upper-bound}--\eqref{eq:Psi-lower-bound},
and the fact that $\Delta_{[\ell+1;r]} (t) = X_{r} (t) - X_{\ell+1} (t)$ from~\eqref{eq:radius-def},
and $F (\eta(t)) = \Psi (X (t))$ by~\eqref{eq:F-def}, 
is that there exists $a>0$ such that, a.s.,  $ a F( \eta(t))  \leq \Delta_{[\ell+1;r]} (t) \leq a^{-1} F( \eta(t))$ for all $t \in \RP$, and hence we obtain~\eqref{eq:exponential-bound}. 

Finally, we deduce~\eqref{eq:log-bound}. Write $\Delta (t) = \Delta_{[\ell+1;r]} (t)$ for short. Let $n_t$ denote the number of times that Markov chain $\eta$ jumps during time interval $[t,t+1]$. Since the jump rate is uniformly bounded, $\Pr [ n_t \geq z ] \leq \re^{-\delta z}$ for some $\delta>0$. Combined with~\eqref{eq:exponential-bound} and the fact that $| \Delta (\tau) - \Delta (0) | \leq B$, a.s., for some $B \in \RP$, this shows that, for a given $\Delta (0)$, we can find $A \in \RP$ large enough so that
\begin{align*}
    \Pr \left[ \sup_{s \in [t,t+1]} \Delta(s) \geq 2 A \log t \right] &
    \leq \Pr [ n_t \geq (A/B) \log t ] + \Pr [ \Delta (t) \geq A \log t ] = O (t^{-2} ), \end{align*}
    say, using~\eqref{eq:exponential-bound}.  The Borel--Cantelli lemma then shows that a.s.,
    $\sup_{t \leq s \leq t+1} \Delta (s) \leq 2 A \log t$  for all but finitely many $t \in \ZP$, which implies~\eqref{eq:log-bound}.
\end{proof}

\section{Proofs: Instability and cloud decomposition}
\label{sec:proofs-instability}

\subsection{Overview}
\label{sec:proofs-instability-overview}

The proof of the cloud decomposition, Theorem~\ref{thm:cloud_decomposition}, is divided into several parts. In the putative cloud decomposition $\fC_1, \ldots, \fC_\nu$, each cloud $\fC_j$ is locally stable (by Proposition~\ref{prop:partial-stability}) and would, were it to be an isolated system with none of the other clouds present, possess the intrinsic speed $-v_j$, with the order $v_1 \geq v_2 \geq \cdots \geq v_\nu$ as at~\eqref{eq:slopes}. The proof of  Theorem~\ref{thm:cloud_decomposition} requires us to establish (i) that these intrinsic cloud speeds are indeed replicated in the full system, and (ii) that clouds are typically well-separated. The structure of the argument to do this is as follows.
\begin{itemize}
    \item We first establish the result for the system in which $v_1 = v_2 = \cdots = v_\nu$. This is the most substantial part of the proof, and requires several steps in itself. In~\S\ref{sec:extremes_separate} we show that the centres of mass of the leftmost and rightmost clouds separate at least diffusively (Proposition~\ref{prop:extreme-clouds}), using a submartingale defocusing result from~\S\ref{sec:submartingale}. 
    \item Then in~\S\ref{sec:multiple_same_speed} we use a sort of induction to show \emph{all} adjacent clouds typically separate, in a quantified sense. Roughly speaking, if there are $\nu \geq 3$ clouds, then since the two extreme clouds spend most of their time far apart, and clouds are tight, there must be another pair of clouds that also spends much time far apart, and so all but a vanishing fraction of the time the system evolves as two independent subsystems, both strictly smaller than the original. Iterating a formal version of this argument gives the result that all clouds typically separate in a quantified sense (Proposition~\ref{prop:cloud-separation-diffusive}).
    \item Having established  separation of adjacent clouds, it is not   hard to show that no cloud's speed can be significantly perturbed from its intrinsic speed, and hence (Lemma~\ref{lem:same-speeds-speeds}) all clouds do indeed travel at the same speed.
    \item Having settled the case of systems where all clouds have the same speed, the general case follows relatively easily by decomposing the system into subsystems consisting of same-speed clouds. Subsystems then have strictly ordered intrinsic speeds which they would maintain as isolated systems (by Lemma~\ref{lem:same-speeds-speeds}) and a coupling argument finishes the proof of the cloud decomposition, particle speeds, and diffusive separation.
    \item The recurrence of adjacent same-speed clouds in~\eqref{eq:recurrence} is established by a separate supermartingale argument via Lemma~\ref{lem:supermartingale}.
\end{itemize}

\subsection{Same-speed clouds: Centre-of-mass dynamics}
\label{sec:extremes_separate}

We start the scheme outlined in~\S\ref{sec:proofs-instability-overview}. 
For $\nu \geq 2$ and $n \in \ZP$, 
define
\begin{equation}
    \label{eq:L-extreme-def}
L (n) :=  \min_{i \in \fC_\nu} X_{i} (\tau_n ) - \max_{i' \in \fC_1} X_{i'} (\tau_n )
= X_{\min \fC_\nu} (\tau_n) - X_{\max \fC_1} (\tau_n), 
\end{equation}
the distance between the leftmost particle in the rightmost cloud and the rightmost particle in the leftmost cloud at time $\tau_n$ defined at~\eqref{eq:tau-sequence}.
The first result of this section states that, in the case where $v_1 = \cdots = v_\nu$,
$L(n)$ grows at least diffusively, meaning that the two extreme clouds are well-separated for all but a vanishing proportion of time
(the case of two particles of the same intrinsic speeds shows that this result can be essentially sharp).

\begin{proposition}
    \label{prop:extreme-clouds}
           Suppose that~\eqref{eq:non-degeneracy} holds, that $\nu \geq 2$, and that
        $v_1 = \cdots = v_{\nu}$. 
 Then for every $\gamma \in (0,1/2)$, there exists $\eps_0 :=\eps_0 (\gamma) >0$ such that, for all $\eps \in (0,\eps_0)$ and all $n$ sufficiently large,  
\[ \Pr \left[ \sum_{i=0}^n \1 { L(i) \leq n^\gamma  } > n^{\frac{1}{2}+\gamma+5\eps} \right] \leq \re^{-n^\eps} .\]
\end{proposition}

The basic element in the proof of Proposition~\ref{prop:extreme-clouds}
is to show that the distance between the extreme clouds (precisely, their centres of mass) is a submartingale. This result, Lemma~\ref{lem:com-submartingale} below, does not need to assume $v_1 = v_\nu$. The intuition for this result is that the two extreme clouds have their own intrinsic
speeds, $v_1$ (for the leftmost) and $v_{\nu}$ (for the rightmost), and, as isolated systems, these intrinsic speeds would be exactly the speeds of the centres of mass of the clouds (cf.~Theorem~\ref{thm:mass-speed}). In the full system, the presence of other particles in between these two clouds should make it harder for them to move towards each other. 

To state the result, let $\fC = [\ell; r]$ for $1 \leq \ell < r \leq N$. Then for $x \in \bbX_N$, define
\begin{equation}
    \label{eq:G-j-def}
    g_{\fC} (x) := \langle \ones, x \rangle_{\ell,r} /  \langle \ones, \ones \rangle_{\ell,r},
    \text{ and, for $t \in \RP$, }
    G_j (t) := g_{\fC_j} ( X(t)) .
\end{equation}
Also  set 
\begin{equation}
    \label{eq:D-n-def}
    \Gamma(n) :=  G_\nu (\tau_n) - G_1 (\tau_n) = g_{[k_{\nu-1}+1;k_\nu]} (X(\tau_n)) - g_{[1;k_1]} ( X(\tau_n)) .
    \end{equation}

\begin{lemma}
\label{lem:com-submartingale}
Suppose that 
$\nu \geq 2$, and recall from~\eqref{eq:slopes} that $v_1 \geq v_\nu$. 
Then, for every $n \in \ZP$, 
\begin{equation}
    \label{eq:com-drift-lower}
 \Exp [ \Gamma(n+1)  - \Gamma(n) \mid \cF_{\tau_n} ] \geq  ( v_1 - v_{\nu}  )/R_{X(\tau_n)} , \as \end{equation}
Suppose, additionally, that~\eqref{eq:non-degeneracy} holds. Then there exists $\delta>0$ such that
\begin{equation}
    \label{eq:com-variance}
 \Exp [ ( \Gamma(n+1)  - \Gamma(n) )^2 \mid \cF_{\tau_n} ] \geq \delta, \as \end{equation}
    \end{lemma}

\begin{remark}[Non-monotonicity]
\label{rem:monotone-1}
It does not appear straightforward to use the intuition for Lemma~\ref{lem:com-submartingale} expressed above directly via  a stochastic monotonicity argument. The exception is in the case where all $m_i \equiv m \in (0,\infty)$ are equal, in which case the natural coupling shows that the presence of an additional particle to the right (say) of a system cannot increase the speed of the centre of mass of the system. For systems with different masses, the non-constant total activity rate  (i.e.~violation of~\eqref{eq:time-change}) means the natural coupling does not work.
\end{remark}

\begin{proof}[Proof of Lemma~\ref{lem:com-submartingale}]
Recall from Definition~\ref{def:majorant} 
and~\eqref{set_of_boundary_indices} that
that $\fC_1 = \{ k_0+1, \ldots,  k_1 \}$  and $\fC_\nu = \{ k_{\nu-1} +1, \ldots, k_{\nu} \}$, with $k_0 = 0$ and $k_{\nu} = N$,
are the leftmost and rightmost clouds, respectively.
    By the definition of~$G_j$ from~\eqref{eq:G-j-def} and Lemma~\ref{lem:drift_y}, 
     with $\ee$ as defined at~\eqref{eq:e-def},
    \begin{align*}
    \Exp [ G_\nu ( \tau_{n+1}  ) - G_\nu (\tau_n )   \mid \cF_{\tau_n} ]
    & = M_{k_{\nu-1},N}^{-1} \Exp_{X(\tau_n)} \langle \ones, \ee \rangle_{k_{\nu-1}+1,N}  \geq - \frac{U_{k_{\nu-1},N}}{M_{k_{\nu-1},N} R_{X(\tau_n)}} , \as \end{align*}
    Similarly,
   \[ \Exp [ G_1 (\tau_{n+1} )  - G_1 (\tau_n) \mid \cF_{\tau_n} ]
    = M_{0,k_1}^{-1} \Exp_{X(\tau_n)} \langle \ones, e \rangle_{1,k_1}
    \leq - \frac{U_{0,k_1}}{M_{0,k_1} R_{X(\tau_n)}} , \as \]
It follows from~\eqref{eq:D-n-def} that (recall that $k_0 = 0$ and $k_\nu = N$)
\[\Exp [ \Gamma(n+1) - \Gamma(n) \mid \cF_{\tau_n} ] 
\geq \frac{1}{R_{X(\tau_n)}} \left( \frac{U_{0,k_1}}{M_{0,k_1}} - \frac{U_{k_{\nu-1},N}}{M_{k_{\nu-1},N}} \right) = \frac{v_1-v_\nu}{R_{X(\tau_n)}}, \as ,\]
    by~\eqref{eq:slopes}, verifying~\eqref{eq:com-drift-lower}. Since $k_1 < k_{\nu-1} +1$, at most one of 
    $G_\nu (\tau_{n+1} ) - G_\nu (\tau_{n}  ) $
    and $G_1 (\tau_{n+1}  ) - G_1 (\tau_{n} ) $ can be non-zero. Hence
    \begin{align*}
         \Exp [ ( \Gamma(n+1)  - \Gamma(n) )^2 \mid \cF_{\tau_n} ] & = 
         \Exp [ ( G_1 (\tau_{n+1} )  - G_1 (\tau_{n}  ) )^2 \mid \cF_{\tau_n} ]  \\
         & {} \quad {} 
         + \Exp [ ( G_\nu (\tau_{n+1} )  - G_\nu ( X(\tau_{n} ) ))^2  \mid \cF_{\tau_n}].
    \end{align*}
The quantity $G_1 (\tau_{n+1} ) - G_1 (\tau_n )$
is non-zero if and only if one of the particles with labels in $\fC_1 = [1;k_1]$ jumps,
and then, by~\eqref{eq:jump-chain} and~\eqref{eq:G-j-def},
     \begin{align*}
    \Exp [ ( G_1 (\tau_{n+1}  ) - G_1 (\tau_n ) ) ^2 \mid \cF_{\tau_n} ]
    & = M_{0,k_1}^{-2} \Exp_{X(\tau_n)} [ \langle \ones, \ee \rangle_{1,k_1}^2 ] \\
    & \geq R_{X(\tau_n)}^{-1} M_{0,k_1}^{-2} 
    \sum_{i \in \fC_1} \bigl( A_i ( X(\tau_n)) +B_i (X(\tau_n)) \bigr) .
    \end{align*}
    Now we use the fact that, by~\eqref{eq:rates}, 
    \[  \sum_{i \in \fC_1}  A_i ( X(\tau_n))  \geq \frac{\min_{j \in \fC_1} m_j}{\max_{j \in \fC_1} m_j} \sum_{i \in \fC_1}  a_i  ,\]
    which, under hypothesis~\eqref{eq:non-degeneracy}, is uniformly positive. Thus~\eqref{eq:com-variance} follows.
    \end{proof}

Recall from~\eqref{eq:G-j-def} that $G_j (t)$
is the centre of mass of cloud $\fC_j$ at time $t \in \RP$.
Define
\begin{equation}
    \label{eq:D-j-def}
   \Gamma_j (n) := G_{j+1} (\tau_n) - G_j (\tau_n), \text{ for } j \in [\nu -1] \text{ and } n \in \ZP, 
\end{equation}
the distance between the centres of mass of clouds~$j$ and $j+1$ at time $\tau_n$. Recall the definition of $L_j (t)$ from~\eqref{eq:L-def}.

\begin{lemma}
\label{lem:supermartingale}
Suppose that $\nu \geq 2$, and that $j \in [\nu-1]$ is such that $v_j = v_{j+1}$, with the notation from~\eqref{eq:slopes}.
Then, for every $n \in \ZP$, 
\begin{equation}
    \label{eq:com-drift-supermart}
 \Exp [ \Gamma_j (n+1)  - \Gamma_j (n) \mid \cF_{\tau_n} ] \leq 0 , \text{ on } \{ L_j (\tau_n) > 0 \}. \end{equation}
\end{lemma}
\begin{proof}
    Suppose that $j \in [\nu-1]$ is such that $v_j = v_{j+1}$. Then $G_j, G_{j+1}$ are the centres of mass
    processes for clouds $\fC_j = \{ k_{j-1}+1, \ldots, k_j\}$ and $\fC_{j+1} = \{ k_j+1,k_{j+1}\}$, respectively. 
From~\eqref{eq:G-j-def},  
\[ \Exp [ G_j (X(\tau_{n+1})) - G_j (X(\tau_n)) \mid \cF_n ] = M_{k_{j-1},k_j}^{-1}
\Exp_{X(\tau_n)} \langle \ones , \ee \rangle_{k_{j-1}+1,k_j} 
\]
    On the event $\{ L_j (\tau_n) > 0 \}$, recalling the definitions of $k_\pm = k_\pm (X (\tau_n))$ from~\eqref{eq:k-plus-minus}, we have $(k_j +1)_- =k_{j}+1$ and $(k_{j})_+ = k_j$. Hence, from~\eqref{eq:drift_y-partial} in Lemma~\ref{lem:drift_y}, we have that, on $\{ L_j (\tau_n) > 0 \}$, 
    \begin{align*} 
    R_{X(\tau_n)} \Exp_{X(\tau_n)} \langle \ones , \ee \rangle_{k_{j}+1,k_{j+1}}  \leq -U_{k_j,k_{j+1}} , \text{ and } 
    R_{X(\tau_n)} \Exp_{X(\tau_n)} \langle \ones , \ee \rangle_{k_{j-1}+1,k_{j}}  \geq -U_{k_{j-1},k_j}
,
\end{align*}
by~\eqref{eq:U-M-diff-def}. 
It follows from~\eqref{eq:D-j-def} that,  on $\{ L_j ( \tau_n) > 0 \}$, 
\begin{align*}
     \Exp [ \Gamma_j (n+1)  - \Gamma_j (n) \mid \cF_{\tau_n} ] &
     = R_{X(\tau_n)}^{-1} \bigl( M_{k_{j},k_{j+1}}^{-1}  \langle \ones , \ee \rangle_{k_{j}+1,k_{j+1}}
     -  M_{k_{j-1},k_j}^{-1}  \langle \ones , \ee \rangle_{k_{j-1}+1,k_{j}} \bigr) \\
&      \leq  R_{X(\tau_n)}^{-1} \left( \frac{U_{k_{j-1},k_j}}{M_{k_{j-1},k_j}} - \frac{U_{k_j,k_{j+1}}}{M_{k_{j},k_{j+1}}}    \right) = 0,
\end{align*}
by~\eqref{eq:slopes} and the assumption that $v_j = v_{j+1}$. This proves~\eqref{eq:com-drift-supermart}.
\end{proof}

Recall the definition of $\Delta_\fC(t)$ from~\eqref{eq:radius-def}.
The following straightforward consequence of Proposition~\ref{prop:partial-stability}
will be useful in showing that to control dynamics of clouds it suffices, on large scales, 
to control dynamics of their centres of mass.

\begin{lemma}
\label{lem:max-cloud-deviation}
        Suppose that  $\nu \in \N$ and
        cloud decomposition $\fC_1, \ldots, \fC_\nu$ has speeds $v_1 = \cdots = v_{\nu} \in \R$.  Then, for every $\eps >0$, for all $n$ sufficiently large,
    \[ \Pr \left[ \max_{j \in [\nu]} \max_{0 \leq i \leq n} \Delta_{\fC_j} ( \tau_i) \geq n^{3\eps} \right] \leq \re^{-n^\eps} .\]
    \end{lemma}
\begin{proof}
    The proof of Proposition~\ref{prop:partial-stability} up to~\eqref{eq:hajek-bound}
    shows that, for every $\eps >0$,
    \[ \max_{j \in [\nu]} \sup_{i \in \ZP} \Pr [ \Delta_{\fC_j} ( \tau_i ) \geq n^{3\eps} ] \leq \re^{-n^{2\eps}} , \]
        for all $n$ sufficiently large. The claimed result follows  from a union bound.
\end{proof}

\begin{proof}[Proof of Proposition~\ref{prop:extreme-clouds}]
    We apply Lemma~\ref{lem:submartingale-occupation} with $X_n = \Gamma(n)$ as defined at~\eqref{eq:D-n-def}. Then the bounded-increments hypothesis~\eqref{eq:bounded-jumps} is satisfied, and Lemma~\ref{lem:com-submartingale} shows that the submartingale hypothesis~\eqref{eq:submartingale} and uniform lower bound on second moments~\eqref{eq:variance-lower} both hold (here we use hypothesis~\eqref{eq:non-degeneracy}).
    Then  Lemma~\ref{lem:submartingale-occupation} shows that, for any $\gamma \in (0,1/2)$ and $\eps \in (0, \frac{1-2\gamma}{4})$,
    \begin{equation}
        \label{eq:D-deviation}
 \Pr \left[ \sum_{i=0}^n \1 { \Gamma(i) \leq n^\gamma } > n^{\frac{1}{2}+\gamma+4\eps} \right] \leq \re^{-n^\eps} .    \end{equation}
    Since $L (n)$ defined at~\eqref{eq:L-extreme-def} satisfies $| L(n) - \Gamma(n)| \leq \Delta_{\fC_1} (\tau_n) + \Delta_{\fC_\nu} (\tau_n)$,
    we combine Lemma~\ref{lem:max-cloud-deviation} with~\eqref{eq:D-deviation} to get
      \begin{align*}
 \Pr \left[ \sum_{i=0}^n \1 { L(i) \leq n^\gamma - 2 n^{3\eps} } > n^{\frac{1}{2}+\gamma+4\eps} \right] & \leq \Pr \left[ \sum_{i=0}^n \1 { \Gamma(i) \leq n^\gamma } > n^{\frac{1}{2}+\gamma+4\eps} \right] \\
 & {} \hspace{-2em} {}  + \Pr \left[ \max_{0 \leq i \leq n} \left( \Delta_{\fC_1} (\tau_i) + \Delta_{\fC_\nu} (\tau_i) \right) \geq 2 n^{3\eps} \right]   \leq 3 \re^{-n^\eps} , \end{align*}
 which yields the result, provided $3\eps < \gamma$.
\end{proof}

\subsection{Same-speed clouds: Quantified separation}
\label{sec:multiple_same_speed}

As outlined in~\S\ref{sec:proofs-instability-overview}, the key result in the proof of Theorem~\ref{thm:cloud_decomposition} is the following quantitative separation result on systems of multiple clouds all of the same intrinsic speed, which shows that not only do the two extreme clouds stay well separated most of the time (as in Proposition~\ref{prop:extreme-clouds}), but \emph{all} adjacent pairs of clouds do.
Recall from~\eqref{eq:L-def} that for $j \in [\nu-1]$, 
$L_j(t) \in \ZP$ denotes the separation of clouds $\fC_j$ and $\fC_{j+1}$ at time $t$.
 
\begin{proposition}
\label{prop:cloud-separation-diffusive}
        Suppose that~\eqref{eq:non-degeneracy} holds, that $\nu \geq 2$, and that
        $v_1 = \cdots = v_{\nu}$. 
 Then there exists $\eps>0$ such that, for all $n$ sufficiently large,
\[ \Pr \left[ \sum_{i=0}^n \1 { \min_{1 \leq j \leq \nu -1} L_j (\tau_i) \leq n^\eps  } > n^{1-\eps} \right] \leq \re^{-n^\eps} .\]
\end{proposition}
\begin{proof}
Suppose $\nu \geq 2$.
For $\gamma \in (0,1/2)$, take $\eps \in (0,\frac{1-2\gamma}{12})$ 
and, for $k \in \N$, define
\begin{equation}
\label{eq:gamma-alpha-k}
\gamma_{k} := \gamma^k,  ~~ \alpha_{k} := 1- 2^{1-k} \gamma^{k(k-1)/2} \left( \frac{1}{2}-\gamma- 6\eps \right) , ~\text{and}~ \eps_k := 2^{1-k} \eps. \end{equation}
This proof works with the jump process $X_{\tau_k}$ and we will abuse terminology slightly and call $\tau_k$ `\emph{time} $k$'.
For
$k \in \{1,2,\ldots,\nu-1\}$, say time~$i \in \ZP$ is $k$-\emph{good}
if $\# \{ j \in [\nu-1] : L_j (i) \geq 2^{-k} n^{\gamma_k} \} \geq k$
(the definition of $k$-good also depends on~$n$, but we assume $n$ is fixed and sufficiently large, as determined in the subsequent argument, and keep the $n$-dependence tacit).
Let $\cG_k(i)$ denote the event that $i$ is $k$-good, and note that $\cG_{k+1} (i) \subseteq \cG_k (i) \in \cF_i$ for $1 \leq k \leq \nu -2$. Define, for $k \in [\nu-1]$, 
\[ B_{k}(A) := \sum_{i \in A} \2{ \cG_k^\rc (i)} , \]
the number of $i \in A \subseteq [n]$ that are not $k$-good.
If $\nu =2$, note that $L_1 (\tau_n) = L(n)$ as defined at~\eqref{eq:L-extreme-def}, and
we 
have 
from Proposition~\ref{prop:extreme-clouds} that, for $\eps>0$ as given, for all $n$ large enough,
\[ \Pr [ B_{1} ( [n] ) \geq n^{\frac{1}{2} +\gamma_1 + 5 \eps} ] \leq \re^{-n^{\eps}}, \]
which
completes the proof if $\nu=2$. The general case is a finite induction, for which we suppose $\nu >2$ and
consider the inductive hypothesis
\begin{equation}
    \label{eq:induction-step}
    \Pr [ B_{k} ( [n] ) \geq n^{\alpha_k} ] \leq \re^{-n^{\eps_k}} , \text{ where }       k \in [\nu-1],
\end{equation}
which we have verified for~$k=1$
with $\alpha_1 = \frac{1}{2}+\gamma_1 + 6 \eps_1$ and $\eps_1 = \eps$ from above, as in~\eqref{eq:gamma-alpha-k}. 

The heart of the induction is to show that, with very high probability, most $k$-good times can be upgraded to $(k+1)$-good times, as follows. If a time $t \in [n]$ is $k$-good, then the next $n_k$ times will have $k$ gaps that remain large (where $n_k$ is chosen appropriately of order $n^{\gamma_k}$), since particles only move to their nearest neighbours, meaning gaps cannot shrink too fast. Over that relatively long time, the $k+1$ subsystems (separated by the gaps that made the original time $k$-good)
evolve independently. As long as $k < \nu-1$, at least one of those subsystems contains at least two clouds, and so we can apply Proposition~\ref{prop:extreme-clouds} to conclude that  the extreme clouds of that subsystem separate for most times in the interval $[t,t+n_k]$. Since clouds are tight by the cloud concentration result in Lemma~\ref{lem:max-cloud-deviation}, this means that there is at least one (additional) large gap between clouds at most of those times.
This additional separation allows us to upgrade most $k$-good times
to $(k+1)$-good times, with high probability. Iterating this shows that, with high probability, most times are $(\nu-1)$-good, which means all clouds are well separated. We give the details, starting by explaining how a particular $k$-good time leads to many $(k+1)$-good times, and then using a blocking argument prove~\eqref{eq:induction-step} for each $1 \leq k \leq \nu -1$.

First we show that, for every $k \in [\nu-2]$, for all $t \in [n]$ and all $n$ large enough, 
\begin{equation}
\label{eq:good-propagates}
\Pr \Bigl[ B_{k+1} \bigl( [t, t+n_k ] \bigr) \geq n^{\alpha_1 \gamma_{k}}   \Bigmid \cF_{\tau_t} \Bigr] \leq \re^{-n^{\eps_1}}, \text{ on } \cG_k (t), \end{equation}
where $n_k := \lfloor 2^{-k-1} n^{\gamma_k} \rfloor$.
Fix any $k$-good time $t \in [n]$. By definition of the $k$-good property, that means there exists $J_t \subset [\nu]$ with size $\# J_t = k$ for which $L_j (\tau_t) \geq 2^{-k} n^{\gamma_k}$ for every $j \in J_t$.
(For definiteness, in case of a choice, choose the $J_t$ that has the smallest possible maximal element.)
Since $|L_j (\tau_{n+1}) - L_j (\tau_n) | \leq 1$ for all $n \in \ZP$,
it follows that
\begin{equation}
\label{eq:k-good} 
L_j (\tau_i) \geq 2^{-k-1} n^{\gamma_k} \text{ for all } i \in [t, t+n_k] \text{ and all } j \in J_t.
\end{equation}

Suppose that $t$ is $k$-good. 
List $J_t$ in order as $J_t = \{ j_1, j_2, \ldots, j_k \}$
and define a corresponding ordered partition of $[\nu]$, denoted $\fB_1, \ldots, \fB_{k+1}$,
by 
\[ \fB_1 := \bigcup_{j=1}^{j_1} \fC_{j} , ~~ \fB_2 := \bigcup_{j=j_1+1}^{j_2} \fC_j , ~~ \ldots ~~, \fB_{k+1} := \bigcup_{j=j_k+1}^{\nu} \fC_j .\]
Each $\fB_i$ is a union of one or more clouds $\fC_j$. Since $k+1 \leq \nu-1$, there is at least one $\fB_\ell = \cup_{j=j_{\ell-1}+1}^{j_\ell} \fC_j$ (with convention $j_0 =0$ and $j_{k+1} = \nu$) that is the union of at least two clouds; choose and fix such $1 \leq \ell \leq k+1$
(the $\fB_i$ and $\ell$ depend on~$t$ but we omit that from the notation). Recall the property~\eqref{eq:k-good}, which states that the collections of clouds $\fB_1, \ldots, \fB_{k+1}$ stay well-separated over time interval $[t, t+ n_k]$. The same is true if we ran a system with the same configuration at time $t$ and independent evolution of subsystems  $\fB_1, \ldots, \fB_{k+1}$. Hence there is a natural coupling over time $[t, t+n_k]$ of the original system to a collection of independent  subsystems  $\fB_1, \ldots, \fB_{k+1}$ (which do not interact). In particular, over time $[t,t+n_k]$, the part of the full system corresponding to $\fB_\ell$, when it jumps, has the same law as an isolated system containing only the particles in $\fB_\ell$.
 
Consider an isolated system containing only the particles corresponding to $\fB_\ell$, started (at time $0$) from the configuration inherited from the full system at time $t$. Denote by $\tau^\ell_0, \tau^\ell_1, \ldots$ the jump times of this isolated system, analogously to~\eqref{eq:tau-sequence}. We retain the original labels for the particles and clusters, and set
\[ L^\ell (n) := \min_{i \in \fC_{j_\ell}} X_i (\tau^\ell_n) - \max_{i' \in \fC_{j_{\ell-1}}} X_{i'} (\tau_n^\ell) , \]
the separation of the two extreme clouds of the isolated system, according to its internal jump clock. 
Proposition~\ref{prop:extreme-clouds} applied to this isolated system 
yields, for $\eps>0$ small enough,
\[ 
\Pr \left[ \sum_{i=0}^{n} 
\1{  L^\ell (i)  \leq n^{\gamma} } > n^{\frac{1}{2}+\gamma+5\eps } 
\right] \leq \re^{-n^\eps} .\]
In addition, write
\[\Delta^\ell (n) := \sum_{j=j_{\ell-1}+2}^{j_\ell-1} \Delta_{\fC_j} (\tau^\ell_n) ,\]
for the total span of all clouds other than the two extreme clouds. An application of Lemma~\ref{lem:max-cloud-deviation} shows that $\Pr [ \max_{0 \leq i \leq n} \Delta^\ell (i) \geq n^{3\eps} ] \leq \re^{-n^\eps}$.
In $n$ steps of the full system, the subsystem corresponding to $\fB_\ell$ takes at most $n$ steps. On the other hand, a binomial concentration bound shows that, with probability at least $1-\re^{-n^\eps}$, the subsystem corresponding to $\fB_\ell$  takes at least $\delta n$ steps, where $\delta>0$ is a positive constant depending only on the $a_i, b_i$, whose existence is guaranteed by~\eqref{eq:non-degeneracy}. 
Hence, since $\gamma \in (0,1/2)$ was arbitrary,
by the coupling described above, we deduce for the full system that, for any $\gamma \in (0,1/2)$, there exists $\eps>0$ so that, for all $n$ large enough, 
\[ 
\Pr \left[ \sum_{i=t}^{t+n} 
\1{  \max_{j_{\ell-1}+1 \leq j \leq j_\ell-1} L_j (\tau_i)  \leq n^{\gamma} } > n^{\alpha_1 } 
\biggmid \cF_{\tau_t} \right] \leq \re^{-n^\eps} , \text{ on } \cG_k (t), \]
where $\alpha_1 = \frac{1}{2} +\gamma +6 \eps$ as at~\eqref{eq:gamma-alpha-k}. 
Combined with~\eqref{eq:k-good}, this verifies~\eqref{eq:good-propagates}.

 Now we can complete the induction. 
Suppose that~\eqref{eq:induction-step} is true for some $k$ with $1 \leq k \leq \nu -2$, $\gamma_k \in (0,1/2)$ and $\alpha_k \in (1/2,1)$ given by~\eqref{eq:gamma-alpha-k}, and some $\eps_k>0$.
Define $s_j := j n_k$, so that $0 = s_0 < s_1 < \cdots < s_{w_k} \leq n$,
where $w_k := \lfloor n/n_k \rfloor$ has $w_k \sim 2^{k+1} n^{1-\gamma_k}$ as $n \to \infty$. Consider time intervals $i_p := [ s_{p-1}, s_p -1 ]$
for $p \in [w_k]$ and $i_{w_k+1} := [ s_{w_k} , n]$.
On the event $B_{k} ( [n] ) < n^{\alpha_k}$,
there can be 
no more than $n^{\alpha_k (1- \gamma_k)}$ intervals~$i_p$ ($p \in [w_k+1]$)
for which $B_{k} ( i_p ) > n^{\alpha_k \gamma_k}$. 
Write $\cP := \{ p \in [w_k+1] : B_k (i_p) \leq n^{\alpha_k \gamma_k} \}$. Then,
since each $i_p$ has size $n_k < n^{\gamma_k}$,
\begin{equation}
\label{eq:Bk-k+1}
B_{k+1} ( [n]) \leq n^{\alpha_k (1-\gamma_k)} n^{\gamma_k} + \sum_{p \in \cP} B_{k+1} ([i_p]) ,
\text{ on } \{ B_k ([n]) < n^{\alpha_k} \}.
\end{equation}
Consider any $p \in \cP$. Then there is some $k$-good time~$t_p$ (for definiteness, choose the earliest) with $s_{p-1} \leq t_p < s_{p}$ and $t_p \leq s_{p-1} + n^{\alpha_k\gamma_k}$.
Hence, by~\eqref{eq:good-propagates} and   $\# \cP = w_k +1 \leq 2^\nu n^{1-\gamma_k}$,
\begin{align*}
 \Pr \Bigl[ \sum_{p \in \cP} B_{k+1} ([i_p]) \geq 2^{1+\nu} n^{1-\gamma_k (1-\alpha_k) } \Bigr] & \leq
2^\nu  n^{1-\gamma_k} \sup_{1 \leq p \leq w_k+1} \Pr [ B_{k+1} ( [i_p] ) \geq 2 n^{\alpha_k \gamma_k} 
] \\
& \leq 2^\nu n^{1-\gamma_k} \sup_{1 \leq p \leq w_k+1} \Pr [ B_{k+1} ( [i_p] ) \geq  n^{\alpha_k \gamma_k}  +n^{\alpha_1 \gamma_k} ] \\
& \leq
\re^{-n^{\eps_1}} , \end{align*}
using also that $\alpha_k \geq \alpha_1$ by~\eqref{eq:gamma-alpha-k}. 
It follows from
the induction hypothesis~\eqref{eq:induction-step} and the fact that $\gamma_k < 1/2$  and $\eps_k \leq \eps_1$ that
\begin{align*}
   & {} \Pr [  B_{k+1} ( [n]) \geq 2^{2+\nu}  n^{1-\gamma_k (1-\alpha_k) } ] \\
    & {} \quad {} \leq  \Pr [  B_{k+1} ( [n]) \geq  2^{1+\nu} n^{1-\gamma_k (1-\alpha_k) } + n^{\alpha_k +\gamma_k(1-\alpha_k)} , \, B_k (n) < n^{\alpha_k} ] + \Pr [ B_k ([n]) \geq n^{\alpha_k} ] \\
      & {} \quad {} \leq  \Pr \Bigl[ \sum_{p \in \cP} B_{k+1} ([i_p]) \geq 2^{1+\nu} n^{1-\gamma_k (1-\alpha_k) } \Bigr] + \re^{-n^{\eps_k}} \leq 2 \re^{-n^{\eps_{k}}},
\end{align*}
 using~\eqref{eq:Bk-k+1}. 
Using the fact that $\alpha_{k+1} > 1 - \gamma_k (1-\alpha_k)$ by~\eqref{eq:gamma-alpha-k} we verify
$\Pr [  B_{k+1} ( [n]) \geq   n^{\alpha_{k+1}} ] \leq \re^{-n^{\eps_{k+1}}}$,
for $\eps_{k+1} = \eps_k/2$ and all $n$ sufficiently large, which completes the inductive step. 
\end{proof}

\begin{lemma}
    \label{lem:same-speeds-speeds}
        Suppose that~\eqref{eq:non-degeneracy} holds  and that
        $v_1 = \cdots = v_{\nu} = v \in \R$.  Then, for every $i \in [N]$,  $\lim_{t\to \infty} t^{-1} X_i (t) = -v$, a.s.
\end{lemma}
\begin{proof}
If $\nu=1$ then the result is contained in Theorem~\ref{thm:stability_condition}, so it suffices to suppose $\nu \geq 2$. Define  $E_k := \{ \min_{1 \leq j \leq \nu-1} L_j (\tau_k) > 0 \}$, the event that at time $\tau_k$ no particles from different clouds are at the same site. 

Consider the $j$th cloud, $\fC_j$, containing particles $k_{j-1}+1 ,\ldots, k_j$.
Recall the definition of $g_{\fC_j}$ and~$G_j$ from~\eqref{eq:G-j-def}.
We will argue similarly to the proof of Theorem~\ref{thm:mass-speed}, which showed that the centre of mass of the whole system always travels at its intrinsic speed; for a stable cloud, the same is true when all the particles in the cloud are free to move, the speed being modified if there is contact with a neighbouring cloud. In particular, 
we have
\[ \Exp_x [ g_{\fC_j} (X(\tau)) - g_{\fC_j} (x) ] = M_{k_{j-1},k_j} \Exp_x \langle \ones, \ee \rangle_{k_{j-1}+1,k_j} .\]
If $x$ is such that $x_{k_{j-1}+1} > x_{k_{j-1}}$ and $x_{k_j+1} > x_{k_j}$, then Lemma~\ref{lem:drift_y}
shows that $R_x \Exp_x \langle \ones, \ee \rangle_{k_{j-1}+1,k_j} = \langle \ones, u \rangle_{k_{j-1}+1,k_j} = - U_{k_{j-1},k_j}$. Hence, on $E_k$,  similarly to~\eqref{eq:com-jump},
\[ \Exp [ G_j (\tau_{k+1}) - G_j (\tau_k) \mid \cF_{\tau_k} ] = -\frac{ U_{k_{j-1},k_j}}{R_{X(\tau_k)}  M_{k_{j-1},k_j}} = -\frac{v_j}{R_{X(\tau_k)}} = -  \frac{v}{R_{X(\tau_k)}},\]
by~\eqref{eq:slopes} and the hypothesis of the lemma.
It follows that there exists $C < \infty$ such that
\[ \left| \Exp [ G_j (\tau_{k+1}) - G_j (\tau_k) \mid \cF_{\tau_k} ] + \frac{v}{R_{X (\tau_k)}} \right| \leq C \2 {E^\rc_k} .\]
Set $A_k := \sum_{\ell=0}^{k-1} R^{-1}_{X(\tau_\ell)}$.
Then following the proof of Theorem~\ref{thm:mass-speed} 
shows that~\eqref{eq:slln-1} holds. By compensating $G_j(\tau_k)$, there is a martingale $M_k = G_j (\tau_k) + v A_k + \zeta_k$, where $|\zeta_k| \leq C \sum_{\ell=0}^k \2 { E_\ell^\rc }$.
Similarly to~\eqref{eq:slln-2}, it can be shown that
\[ \lim_{k \to \infty} \frac{1}{k} \left| G_j (\tau_k) + v A_k \right| \leq \limsup_{k \to \infty} \frac{C}{k} \sum_{\ell=0}^k \2{ E^\rc_\ell} .\]
 But here,
Proposition~\ref{prop:cloud-separation-diffusive} shows that
$\lim_{k \to \infty} k^{-1} \sum_{\ell=0}^k \2{ E^\rc_\ell} =0$, a.s., and so, using~\eqref{eq:slln-1}, as in the proof of Theorem~\ref{thm:mass-speed}, we conclude that $\lim_{k \to \infty} \tau_k^{-1} G_j (\tau_k) = -v$, a.s.
  Also by Lemma~\ref{lem:max-cloud-deviation}, for every $i \in \fC_j$, $\tau_k^{-1} | X_i (\tau_k) - G_j (\tau_k) | \to 0$, a.s. Thus $\lim_{k\to\infty} \tau_k^{-1} X_i (\tau_k) = -v$, a.s., and an interpolation
  argument similar to the one at the end of the proof of  Theorem~\ref{thm:mass-speed} shows that
  $\lim_{t \to \infty} t^{-1} X_i (t) = -v$, a.s., as claimed.
\end{proof}

Finally, we can complete the proof of the cloud decomposition, Theorem~\ref{thm:cloud_decomposition}.

\begin{proof}[Proof of Theorem~\ref{thm:cloud_decomposition}]
The cloud stability result, part~\ref{thm:cloud_decomposition-b}, is a consequence of Proposition~\ref{prop:partial-stability}. 
For the remaining parts, we use a  coupling. 
Split the system into subsystems of constant-speed clouds,
so for example, $\fB_1 = \fC_1 \cup \cdots \cup \fC_{i_1}$ is the leftmost subsystem, where $i_1 = \max \{ j \in [\nu] : v_j = v_1\}$. 
This gives subsystems $\fB_1, \ldots, \fB_K$, say $(K \in [\nu])$.
Consider a system that starts from the same initial configuration as the original system, 
but in which the subsystems $\fB_j$ evolve independently of each other. In the 
independent system,  each $\fB_j$ evolves as a system with same-speed clouds, and so satisfies the strong law from Lemma~\ref{lem:same-speeds-speeds} and the cloud separation result from Proposition~\ref{prop:cloud-separation-diffusive}. We couple the independent system and the original system such that whenever the original system is in a configuration in which no two subsystems have particles at the same site, we use the same holding times and jumps as the independent system. By the strict ordering of the speeds, in particular, the strong law shows that there is an a.s.-finite time $\tau$ for which, for all $t \geq \tau$, the two systems evolve with exactly the same increments. It follows that the maximum displacement between the particles in the original system and their partners in the independent system remains bounded by a finite random variable. This verifies the speeds for the original system as stated in part~\ref{thm:cloud_decomposition-a} of the theorem, and also
means that the separation result from Proposition~\ref{prop:cloud-separation-diffusive}
remains valid for the original system, which proves statement~\eqref{eq:same-speed-separation} in part~\ref{thm:cloud_decomposition-c} of the theorem.
It remains to prove~\eqref{eq:recurrence}. Fix $s \geq 0$ and let $\kappa_s := \inf\{t \geq s : L_j (t) = 0\}$. Then Lemma~\ref{lem:supermartingale} shows that $\Gamma_j (t \wedge \kappa_s)$ is a non-negative supermartingale, which hence converges a.s.~as $t \to \infty$. The only way this can happen is if $\kappa_s < \infty$, a.s. Since $s \geq 0$ was arbitrary, we conclude~\eqref{eq:recurrence}.
\end{proof}

\section{Comments on invariant distributions}
\label{sec:invariant_distributions}

\subsection{Irreversibility of dynamics}
\label{sec:irreversibility}

A continuous-time Markov chain on a countable state space $\bbX$, with transition
rates $q_{i,j}$, $i,j \in \bbX$, is \emph{reversible}
with respect to an invariant measure $\pi = (\pi_i)_{i \in \bbX}$ if the detailed balance relation
$\pi_i q_{i,j} = \pi_j q_{j,i}$ holds for all $i, j \in \bbX$. For any
cycle of states $i_1, i_2, \ldots, i_n, i_{n+1}$ with $i_1 = i_{n+1}$ and no other pair of states equal,
taking a product of the detailed balance relations $\pi_{i_k} q_{i_k, i_{k+1}} = \pi_{i_{k+1}} q_{i_{k+1}, i}$ over $k \in \{1,\ldots, n\}$ implies that 
\begin{equation}
    \label{eq:kolmogorov}
 \prod_{k=1}^{n} q_{i_k, i_{k+1}}  = 
\prod_{k=1}^{n} q_{i_{k+1},i_{k}} ,
\end{equation}
i.e., the product of the transition rates along the forward cycle is the same as the 
 product of the transition rates along the backward cycle.
In fact, a criterion of Kolmogorov says that the Markov chain is reversible if and only if~\eqref{eq:kolmogorov} holds for \emph{every} cycle~\cite[p.~23]{kelly}.

It is well known that the simple exclusion process, as described in \S\ref{sec:exclusion}, is reversible,
and this
is key to the explicit computation of invariant measures. For the elastic model
studied in the present paper, the generic situation, provided $N \geq 3$, is that reversibility fails:
see Figure~\ref{fig:cycle} for an example. Thus it seems that, generically, it may be difficult to obtain an explicit expression for the invariant measure in the stable case.

\begin{remark}[Discrete vs.~continuous time]
Consider the elastic particle system with
$m_i \equiv m \in (0,\infty)$ for all $i \in [N]$.
As mentioned in \S\ref{sec:discrete-time}, since the total activity rate~\eqref{eq:time-change} of the elastic model with constant masses is independent of the current configuration, the discrete-time jump chain associated with the continuous-time Markov chain is stable precisely when the continuous-time chain is stable, and the stationary measures coincide. The Kolmogorov cycle criterion has a direct discrete-time analogue, where in~\eqref{eq:kolmogorov} one replaces the transition rates by transition \emph{probabilities}. Since the total activity rate is constant, the same example in Figure~\ref{fig:cycle} also shows that reversibility fails in the discrete-time setting, as it must. By contrast, the exclusion process does not have constant total activity rate, but nevertheless it turns out that the jump-chain of the exclusion process is \emph{also} reversible (with a different invariant measure from the continuous-time version).
\end{remark}

\begin{figure}[t]
\centering
\scalebox{0.85}{
 \begin{tikzpicture}[domain=0:1, scale=1.2]
\draw[dotted,<->] (0,0) -- (10,0);
\node at (10.4,0)       {$\Z$};
\draw[black,fill=white] (1,0) circle (.5ex);
\draw[black,fill=black] (2,0) circle (.5ex);
\draw[black,fill=white] (3,0) circle (.5ex);
\draw[black,fill=white] (4,0) circle (.5ex);
\draw[black,fill=black] (5,0) circle (.5ex);
\draw[black,fill=white] (6,0) circle (.5ex);
\draw[black,fill=white] (7,0) circle (.5ex);
\draw[black,fill=black] (8,0) circle (.5ex);
\draw[black,fill=white] (9,0) circle (.5ex);
\node at (1.3,0.6)       {\small $1$};
\node at (2.5,0.6)       {\small $2$};
\draw[black,->,>=stealth] (2,0) arc (30:141:0.58);
\draw[black,->,>=stealth] (5,0) arc (30:141:0.58);
\node at (4.5,0.6)       {\small $1$};
\node at (5.5,0.6)       {\small $1$};
\draw[black,->,>=stealth] (8,0) arc (30:141:0.58);
\node at (7.5,0.6)       {\small $2$};
\node at (8.5,0.6)       {\small $1$};
\draw[black,->,>=stealth] (2,0) arc (150:39:0.58);
\draw[black,->,>=stealth] (5,0) arc (150:39:0.58);
\draw[black,->,>=stealth] (8,0) arc (150:39:0.58);
\end{tikzpicture}}
\bigskip
\bigskip

\scalebox{0.57}{
 \begin{tikzpicture}[-,auto,node distance=3cm,
  thick,main node/.style={rectangle,draw,inner sep=0pt,minimum size=14pt,very thick},body node/.style={rectangle,draw,inner sep=10pt,minimum size=14pt,very thick},fake node/.style={rectangle,inner sep=10pt}]
\node[fake node] (10) at (4,5) {\usebox{\configaa}};
\node[fake node] (11) at (4,4.5) {\usebox{\configaa}};
\node[fake node] (12) at (5,4.5) {\usebox{\configaa}};
\node[fake node] (30) at (7.5,0.5) {\usebox{\configaa}};
\node[fake node] (31) at (7.5,-0.5) {\usebox{\configaa}};
\node[fake node] (4) at (3,4.5) {\usebox{\configaa}};
\node[fake node] (41) at (4,4.5) {\usebox{\configaa}};
\node[fake node] (5) at (-1,0.5) {\usebox{\configaa}};
\node[fake node] (20) at (0.5,0.5) {\usebox{\configaa}};
\node[fake node] (21) at (0.5,-0.5) {\usebox{\configaa}};
\node[fake node] (22) at (0,0.5) {\usebox{\configaa}};
\node[fake node] (23) at (8.0,0.5) {\usebox{\configaa}};
\node[fake node] (24) at (9,0.5) {\usebox{\configaa}};
\node[body node] (1) at (4,5) {\usebox{\configa}};
\node[body node] (2) at (0,0) {\usebox{\configb}};
\node[body node] (3) at (8,0) {\usebox{\configc}};
    \path[->]
(4) edge node [sloped,anchor=center,above] {$1$} (5);
\path[->]
(22) edge node [sloped,anchor=center,below] {$3$} (41);
\path[->]
(30) edge node [sloped,anchor=center,above] {$1$} (20);
\path[<-]
(31) edge node [sloped,anchor=center,below] {$2$} (21);
\path[<-]
(12) edge node [sloped,anchor=center,above] {$2$} (24);
\path[->]
(11) edge node [sloped,anchor=center,below] {$2$} (23);
\end{tikzpicture}}
\hfill
\scalebox{0.57}{
 \begin{tikzpicture}[-,auto,node distance=3cm,
  thick,main node/.style={rectangle,draw,inner sep=0pt,minimum size=14pt,very thick},body node/.style={rectangle,draw,inner sep=10pt,minimum size=14pt,very thick},fake node/.style={rectangle,inner sep=10pt}]
\node[fake node] (10) at (4,5) {\usebox{\configaa}};
\node[fake node] (11) at (4,4.5) {\usebox{\configaa}};
\node[fake node] (12) at (5,4.5) {\usebox{\configaa}};
\node[fake node] (30) at (7.5,0.5) {\usebox{\configaa}};
\node[fake node] (31) at (7.5,-0.5) {\usebox{\configaa}};
\node[fake node] (4) at (3,4.5) {\usebox{\configaa}};
\node[fake node] (41) at (4,4.5) {\usebox{\configaa}};
\node[fake node] (5) at (-1,0.5) {\usebox{\configaa}};
\node[fake node] (20) at (0.5,0.5) {\usebox{\configaa}};
\node[fake node] (21) at (0.5,-0.5) {\usebox{\configaa}};
\node[fake node] (22) at (0,0.5) {\usebox{\configaa}};
\node[fake node] (23) at (8.0,0.5) {\usebox{\configaa}};
\node[fake node] (24) at (9,0.5) {\usebox{\configaa}};
\node[body node] (1) at (4,5) {\usebox{\configa}};
\node[body node] (2) at (0,0) {\usebox{\configb}};
\node[body node] (3) at (8,0) {\usebox{\configc}};
    \path[->]
(4) edge node [sloped,anchor=center,above] {$1$} (5);
\path[->]
(22) edge node [sloped,anchor=center,below] {$1$} (41);
\path[->]
(30) edge node [sloped,anchor=center,above] {$1$} (20);
\path[<-]
(31) edge node [sloped,anchor=center,below] {$2$} (21);
\path[<-]
(12) edge node [sloped,anchor=center,above] {$1$} (24);
\path[->]
(11) edge node [sloped,anchor=center,below] {$2$} (23);
\end{tikzpicture}}
\caption{An example of an $N=3$ particle system showing that the dynamics for the elastic process is not reversible. Take jump rates $(a_i, b_i)$ for particles in free space given by $(1,2), (1,1), (2,1)$ (\emph{top picture}). Pictured are the transition rates for the cycle $(1,1) \leftrightarrow (0,2) \leftrightarrow (0,1)
\leftrightarrow (1,1)$ in the elastic model (\emph{bottom left}) as opposed to the exclusion process (\emph{bottom right}). In the exclusion process the rates for sequence
$(1,1) \rightarrow (0,2) \rightarrow (0,1) \rightarrow (1,1)$ 
and its reversal
$(1,1) \leftarrow (0,2) \leftarrow (0,1) \leftarrow (1,1)$ 
have the same product: $1 \cdot 2 \cdot 1 = 2 \cdot 1 \cdot 1$. In the elastic  process, the rates for sequence
$(1,1) \rightarrow (0,2) \rightarrow (0,1) \rightarrow (1,1)$ 
have product $1 \cdot 2 \cdot 2 = 4$ while  its reversal
$(1,1) \leftarrow (0,2) \leftarrow (0,1) \leftarrow (1,1)$ 
has product $2 \cdot 1 \cdot 3 = 6$.}
\label{fig:cycle}
\end{figure}
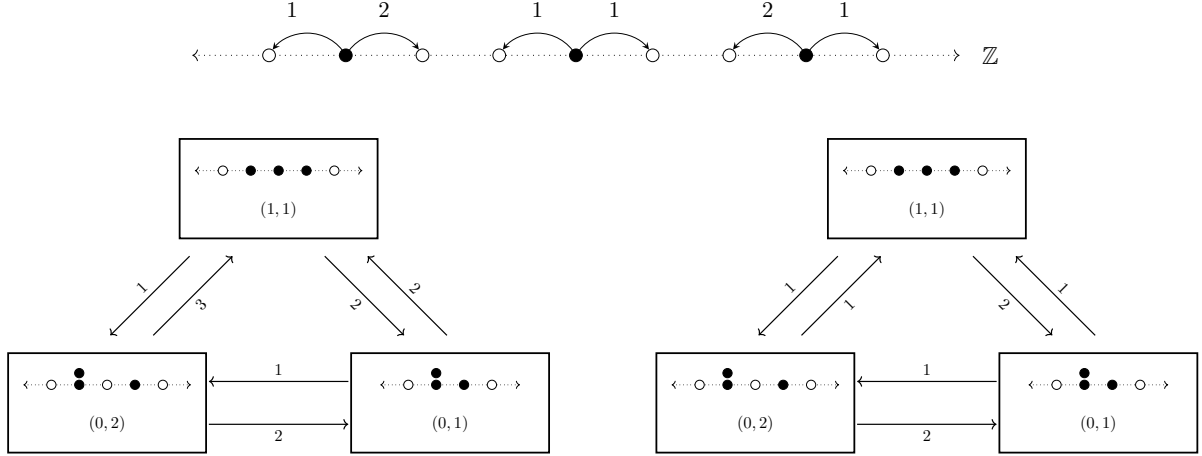

\subsection{Exact computations}
\label{sec:exact_computations}
 
Suppose that the elastic particle system is stable, i.e.,
satisfies the criterion~\eqref{eq:stability_condition} from Theorem~\ref{thm:stability_condition}. Then the Markov process
of inter-particle separations $\eta (t) := (\eta_1(t), \ldots, \eta_{N-1}(t))$ has an invariant distribution $\pi^{(N-1)} (x)$, $x \in \ZP^{N-1}$ (cf.~Corollary~\ref{cor:stability}).
Reversibility often enables one to explicitly compute the corresponding invariant measure;
for example, the exclusion process (\S\ref{sec:exclusion}) is reversible and stationary measures are product-geometric (see e.g.~\cite{mmpw}). 
For non-reversible systems exact computation of stationary distributions is typically much harder. Generically, as in Figure~\ref{fig:cycle}, there is no reversibility for the elastic particle system for $N \geq 3$, and exact computation of stationary distributions remains an open problem (we comment on this in~\S\ref{sec:problems} below). However, the case $N=2$ \emph{is} reversible (since it reduces to a birth-death random walk on $\ZP$). 
In this section we record this special case, and we finish with some open problems in~\S\ref{sec:problems}, including to identify
  exceptional parameter values for the general $N \geq 3$ case where   reversibility holds. 

\paragraph*{Case \texorpdfstring{$N=2$}{N=2}.} 
For two particles (see Theorem~\ref{thm:stability_condition} and Example~\ref{ex:small-systems}) the system is stable if and only if $u_1 > u_2$,
where $u_i$ is as defined at~\eqref{eq:intrinsic-speed} 
(note that this implies $b_1 + a_2 >0$). The inter-particle distance is a birth-death random walk on $\ZP$ and, in the positive-recurrent case,
the stationary distribution $\pi : \ZP \to [0,1]$ can be computed explicitly
by solving 
\begin{align*}  
(a_1+a_2+b_1+b_2) \pi_0 & =  (b_1+a_2) \pi_1, \\
(a_1+a_2+b_1+b_2) \pi_1 & =  (b_1+a_2) \pi_2 + (a_1+a_2+b_1+b_2) \pi_0 \\
(a_1+a_2+b_1+b_2) \pi_k & =  (b_1+a_2) \pi_{k+1} +  (a_1+b_2) \pi_{k-1}, ~k \geq 2. 
\end{align*}
The solution obtained (cf.\ e.g.~\cite[pp.~197--8]{anderson}) is
a \emph{zero-modified geometric} distribution
\begin{align*} \pi_0   = \frac{u_1-u_2}{2 (b_1+a_2)} , \text{ and }
\pi_k  = \frac{(u_1-u_2) (a_1+a_2+b_1+b_2) (a_1+b_2)^{k-1}}{2 (b_1+a_2)^{k+1}}, ~k \geq 1.
\end{align*}
An alternative argument uses speeds and `ergodic' considerations. By stability, the long-run average speed of any particle must be the same as that of the centre of mass of the system, which, by Theorem~\ref{thm:mass-speed}, is always $-U_2/M_2 = (u_1+u_2)/2$.
The leftmost particle travels at its intrinsic speed $u_1$ unless $\eta_1 (t) = 0$, in which case it cannot jump to the right but its speed of jumping left is increased. It is not hard to show then that
\[ \frac{u_1+u_2}{2} = u_1 - ( b_1+a_2) \lim_{t \to \infty} \frac{1}{t} \int_0^t \1 { \eta_1(s) = 0 } \ud s 
= u_1 -  ( b_1+a_2) \pi_0 , \]
recovering the formula for $\pi_0$; considering the rightmost particle gives the same conclusion.

\paragraph*{Case \texorpdfstring{$N \geq 3$}{N>2}.} 
Without reversibility, it appears challenging to proceed in general.
Even using the `ergodic' property and speeds, as in the last example, to compute stationary probabilities of one or more coordinates being zero does not seem straightforward: 
considering the speeds gives $N+1$ equations and there are $2^{N-1}$ unknowns (each coordinate can be zero or non-zero), but, as in the above example, the $N+1$ equations are not linearly independent, so even for $N =3$ the system is undetermined.

\subsection{Open problems}
\label{sec:problems}

It is known in the setting of Atlas models (see~\S\ref{sec:atlas} and~\cite{pp}) that  the \emph{skew symmetry} condition of Harrison \& Williams \cite{harrison,harrison-williams,williams}, transferred from the theory of reflected diffusions in the orthant, yields a necessary and sufficient condition on parameters of the model in order for it to have a product-exponential invariant distribution, and possess a 
relative of the reversibility property known as \emph{strong duality}. We suspect that there should be a related picture in our setting, given the strong parallels with the Atlas model, but it is not clear what condition is needed to play the role of skew symmetry: the Atlas model is a certain heavy-traffic scaling limit of the particle system (see~\cite{kps} and~\cite[\S6.3]{mmpw}) but there are many potential conditions on the $a_i, b_i$ that can lead to the limiting parameters satisfying the skew symmetry condition, and no candidate that we tried was adequate to guarantee reversibility.

 \begin{problem}
For  $N \geq 3$, formulate a necessary and sufficient condition on the $a_i, b_i$ (and $m_i$) under which
the elastic particle system is reversible and has a product-geometric invariant measure.
 \end{problem}

In the case $N=3$ the problem is equivalent (see~\S\ref{sec:queues}) to that of a nearest-neighbour random walk on $\ZP^2$ with boundary reflections. 
It remains a challenging problem to obtain (necessary and/or sufficient) conditions under which a random walk on $\ZP^2$ has a stationary measure that can be expressed explicitly, with particular interest on expressions as a product of geometric terms, or a finite or countable sum of such products: see e.g.~\cite{adan,chen2015,adan2023,chen2016} and references therein. The scope of the \emph{compensation approach} as described in~\cite{adan} appears to cover the $N=3$ case of the elastic model, but successful application of that approach to this setting remains an open problem, as far as we are aware.
There are some classes of walks that have been solved, for example in~\cite{chen2015,chen2016,adan2023}, but the forms of boundary reflections assumed 
in~\cite{chen2015,chen2016,adan2023} exclude application to any stable regime for the present model.

 \begin{problem}
More generally, compute explicitly the invariant measure of the elastic particle system when $N \geq 3$.
 \end{problem}

\appendix

\section{Submartingale occupation and defocusing}
\label{sec:submartingale}

For $X:= (X_n)_{n \in \ZP}$, a stochastic process taking values in $\RP$,
we define
\begin{equation}
    \label{eq:local-time}
    \cL^X_n (x) := \sum_{i=0}^n \1{ X_i \leq x },
    \text{ for } x \in \RP, 
\end{equation}
the occupation time of interval~$[0,x]$ up to time $n \in \ZP$.

\begin{lemma}
\label{lem:submartingale-occupation}
    Let $X = (X_n)_{n \in \ZP}$ be an $\RP$-valued process, adapted to
    a filtration $(\cF_n)_{n \in \ZP}$. Suppose that there exist $B \in \RP$ and $v \in (0,\infty)$
    such that, for all $n \in \ZP$,
    \begin{align}
    \label{eq:bounded-jumps}
    \Pr [ X_{n+1} - X_n  \leq B ] & = 1; \\
        \label{eq:submartingale}
        \Exp [ X_{n+1} - X_n \mid \cF_n ] & \geq 0, \as; \\
        \label{eq:variance-lower}
        \Exp [ (X_{n+1} - X_n)^2 \mid \cF_n ] & \geq v, \as
    \end{align}
Then, for any $\gamma \in (0,1/2)$ and any $\eps \in (0, \frac{1-2\gamma}{4})$, for all $n$ sufficiently large, 
\[ \Pr \Bigl[ \cL^X_n (n^{\gamma}) \geq n^{\frac{1}{2} + \gamma + 4\eps} \Bigr] \leq \re^{-n^\eps} ,
\text{ and } \Exp \cL^X_n (n^{\gamma}) \leq n^{\frac{1}{2} + \gamma + 4\eps} . \]
\end{lemma}
\begin{remarks}
    \label{rems:submartingale-occupation}
    \begin{myenumi}[label=(\roman*)]
    \item \label{rems:submartingale-occupation-i}
        A consequence of Lemma~\ref{lem:submartingale-occupation} (taking $\gamma \approx 0$ and then $\gamma \approx 1/2$) and the Borel--Cantelli lemma is that,  for every $\eps >0$ and every $B \in \RP$, 
   \begin{equation}
        \label{eq:diffusive-separation}
    \lim_{n \to \infty} n^{-1} \cL_n^X \bigl( n^{(1/2)-\eps} \bigr) 
    = \lim_{n \to \infty} n^{-(1/2)-\eps}   \cL_n^X ( B )  = 0, \as,
    \end{equation}
    which says, roughly speaking, that $X$  is (i) for all but a vanishing proportion of time growing almost diffusively, and  (ii) only of constant size for a roughly diffusively growing cumulative time. 
This `diffusive' behaviour is (to within $\log$ factors) sharp in the case of one-dimensional symmetric simple random walk,
 and so no significantly stronger conclusion is valid in general. 
    \item \label{rems:submartingale-occupation-ii}
    If $M_0, M_1, \ldots$ is a martingale with uniformly bounded increments and $\Exp [ (M_{n+1} - M_n)^2 \mid \cF_n ] \geq v$, a.s.,
    then the hypotheses~\eqref{eq:bounded-jumps}--\eqref{eq:variance-lower} are satisfied by the submartingale $X_n = |M_n|$. Thus Lemma~\ref{lem:submartingale-occupation} can be viewed in relation to results on ``martingale defocusing''~(see \cite{pss,ggpz,az}).
     In that context it is known~(see \cite{ggpz}) that, for example, one can construct martingales of this type for which $\Pr [ |M_n| \leq B ] \geq n^{-\beta}$ for $\beta \in (0,1/2)$ that can be arbitrarily small,
    for \emph{very specific} times~$n$ (indeed, the construction of the martingale in~\cite{ggpz} depends on the specific choice of~$n$). Our result, however, shows that in an averaged sense the majority of times conform with the intuition that such a martingale should be `diffusive', since Lemma~\ref{lem:submartingale-occupation} implies that $\sum_{i=0}^n \Pr [ |M_i| \leq B ] = O (n^{(1/2)+\eps})$.
\end{myenumi}
\end{remarks}

\begin{proof}[Proof of Lemma~\ref{lem:submartingale-occupation}]
Let $x \in (0,\infty)$. 
    Define $\lambda^x_0:=0$ and then, successively, for $k \in \N$,
    \begin{align*}
        \rho^x_{k} & := \inf \{ m \geq \lambda^x_{k-1} : X_m \geq 2 x\}, \text{ and }
        \lambda^x_{k}  := \inf \{ m \geq \rho^x_{k} : X_m \leq x \},
    \end{align*}
    where $\inf \emptyset := \infty$, as usual. 
    To lighten the notation, we drop the superscript $x$ from now on, and write simply $\lambda_k, \rho_k$. 
    Then $0 = \lambda_0 \leq \rho_1 \leq \lambda_1 \leq \cdots$ are (possibly infinite) stopping times that partition the path of the process
    into \emph{high-level} excursions over time intervals $[\rho_k, \lambda_k)$
    and \emph{low-level} excursions over time intervals $[\lambda_k, \rho_{k+1})$. Any visits by the process to $[0,x]$ must occur during low-level excursions.
    
    We show that
    low-level excursions are, cumulatively, much shorter in duration than high-level excursions. First note that
    Theorem~2.4.12 of~\cite[p.~45]{mpw-book}
    with  hypotheses~\eqref{eq:submartingale}--\eqref{eq:variance-lower}
    shows that, for every $\ell \in \ZP$,
    on $\{ \lambda_k < \infty\}$, 
\[
\Pr \Bigl[ \max_{\lambda_k +\ell n \leq m \leq \lambda_k + (\ell+1) n} X_m \leq 2x \Bigmid \cF_{\lambda_k+\ell n} \Bigr]
\leq \frac{(2x + B)^2}{vn} , 
    \]
    by hypothesis~\eqref{eq:bounded-jumps}.
    Hence there is a constant $r >0$ (depending only on $B$ and $v$) for which 
    \begin{align*}
\Pr [ \rho_{k+1} - \lambda_k \geq (\ell+1) \lfloor r x^2 \rfloor \mid \cF_{\lambda_k +\ell \lfloor r x^2 \rfloor} ]
\leq \frac{1}{\re} \1{ \rho_{k+1} -\lambda_k \geq \ell \lfloor r x^2 \rfloor },     
    \end{align*}
    from which we deduce that,  on $\{ \lambda_k < \infty\}$,
     \begin{equation}
    \label{eq:compact-excursions}
\Pr [  \rho_{k+1} - \lambda_k \geq \ell r x^2 \mid \cF_{\lambda_k} ] \leq \re^{-\ell}.
    \end{equation}
 From~\eqref{eq:compact-excursions}
   there exists $C < \infty$ (depending only on~$v$ and~$B$)
    such that, for any $\eps>0$,
    \begin{align}
    \label{eq:low-level}
     &    \Pr \left[ \sum_{k=0}^{n-1} ( \rho_{k+1} - \lambda_k )   > C n^{1+2\eps} x^2 \right] \nonumber\\
        & {} \quad {} \leq \Pr \left[ \bigcup_{k=0}^{n-1} \left\{ ( \rho_{k+1} - \lambda_k )  \1 {\lambda_k < \infty }  > C n^{2\eps} x^2 \right\} \right]       \leq n \re^{-n^{2\eps}} ,
    \end{align}
    where we set $\rho_{k+1} - \lambda_k =0$ if $\lambda_k = \infty$.
It follows from~\eqref{eq:compact-excursions} that
    $\{ \lambda_k < \infty \} \subseteq \{ \rho_{k+1} < \infty \}$, up to sets of
    probability zero. Considering high-level excursions,    
    Lemma~2.7.7 of~\cite[p.~76]{mpw-book} shows that there exists $\delta>0$ such that, for $k \in \N$, on $\{ \lambda_{k-1} < \infty \}$,
    \begin{equation}
    \label{eq:high-level-1}
    \Pr \bigl[ \lambda_k - \rho_k \geq 4 \delta y^2 \bigmid \cF_{\rho_k} \bigr]
    \geq \frac{1}{2} \Pr \Bigl[ \max_{\rho_k \leq m \leq \lambda_k} X_m \geq 2(x+y) \Bigmid \cF_{\rho_k} \Bigr] .
    \end{equation}
    For $y >  B$, let $\tau_k (y) := \inf \{ m \geq \rho_k : X_m \geq 2(x+y) \}$;
    by the same argument as~\eqref{eq:compact-excursions}, one has  $\{ \lambda_{k-1} < \infty \} \subseteq \{ \tau_k (y) < \infty \}$. 
    Using hypothesis~\eqref{eq:submartingale}, the submartingale optional stopping theorem (e.g.~Theorem~2.3.7 of~\cite[p.~33]{mpw-book})
    implies that, on $\{ \lambda_{k-1} < \infty \}$,
    \begin{align*}
    2x \leq X_{\rho_k} & \leq \Exp [ X_{\lambda_k \wedge \tau_k (y)} \mid \cF_{\rho_k} ] \\
&    \leq x \Pr [ \lambda_k < \tau_k (y) \mid \cF_{\rho_k} ] + ( 2(x+y)+B) \Pr [ \lambda_k > \tau_k (y) \mid \cF_{\rho_k} ],
    \end{align*}
    since, by~\eqref{eq:bounded-jumps}, $X_{\tau_k (y)} \leq 2(x+y)+B$ whenever $\tau_k(y) < \infty$. 
    It follows that, on $\{ \lambda_{k-1} < \infty \}$,
       \begin{equation}
    \label{eq:high-level-2} 
    \Pr \Bigl[ \max_{\rho_k \leq m \leq \lambda_k} X_m \geq 2(x+y) \Bigmid \cF_{\rho_k} \Bigr] =
     \Pr [ \tau_k (y) < \lambda_k   \mid \cF_{\rho_k} ] \geq \frac{x}{x+2y+ B}.
    \end{equation}
    Combining~\eqref{eq:high-level-1}--\eqref{eq:high-level-2}, we obtain 
     \begin{equation}
    \label{eq:high-level-3} 
     \Pr \bigl[ \lambda_k - \rho_k \geq 4  \delta y^2 \bigmid \cF_{\rho_k} \bigr] \geq \frac{x}{2x+4y+ 2B}, \text{ on } \{ \lambda_{k-1} < \infty \}.
    \end{equation}
    Then, denoting $M_n := \max_{1 \leq k \leq n} ( \lambda_k - \rho_k )$,
    \begin{align*}
        \Pr \left[ M_n < 4  \delta y^2 \right]
        & \leq \Exp \bigl[ \Pr  [ M_n < 4  \delta y^2 \mid \cF_{\rho_{n}} ] \1 { \lambda_{n-1} < \infty } \bigr]
        \\
        & = \Exp \bigl[ \Pr  [ \lambda_n - \rho_n < 4  \delta y^2 
          \mid \cF_{\rho_{n}} ] \1 { \lambda_{n-1} < \infty }  \1 { M_{n-1} < 4  \delta y^2 } \bigr] \\
          & \leq \left( 1 - \frac{x}{2x+4y+ 2B} \right)  \Pr \left[ M_{n-1} < 4  \delta y^2 \right],
    \end{align*}
    using~\eqref{eq:high-level-3} in the final step. 
    This recursion leads to the bound
 \begin{equation}
    \label{eq:high-level-4} 
    \Pr \left[ \sum_{k=1}^{n} ( \lambda_{k} - \rho_k ) \geq 4  \delta y^2  \right] 
    \geq 
    \Pr \left[ M_n \geq 4  \delta y^2 \right] 
    \geq 1 -  \left( 1 - \frac{x}{2x+4y+ 2B} \right)^n .
    \end{equation}
    Fix $\gamma \in (0,1/2)$, $\eps \in (0, \frac{1-2\gamma}{4})$, and $\alpha = \frac{1}{1-2\gamma + 4\eps}$. Choose $x=n^{2\alpha \gamma}$ and $y = n^\alpha/(2\sqrt{\delta})$. Then~\eqref{eq:low-level} and~\eqref{eq:high-level-4}
    combine to show that, for some $c>0$ and all $n \in \N$ sufficiently large, 
    \begin{align*} 
     \Pr \left[ \sum_{k=0}^{n-1} ( \rho_{k+1} - \lambda_k ) \leq C n^{1+3 \alpha \eps+4\alpha \gamma}, \, \sum_{k=1}^{n} ( \lambda_{k} - \rho_k ) \geq  n^{2\alpha }  \right] 
     & \geq 1 - n \re^{-n^{3\alpha \eps}} - \re^{-c n^{1- \alpha (1-2\gamma)}} \\
     & \geq 1 - \re^{-n^{2\alpha \eps}},
    \end{align*}
    since $1 - \alpha (1-2\gamma) > 3 \alpha \eps$.
    But the latter event implies that $\cL^X_{\lfloor n^{2\alpha} \rfloor} ( n^{2 \alpha \gamma} ) \leq C n^{1+3 \alpha \eps+4\alpha \gamma}$.
    After a change of variables, this means that, for all $n$ large enough, 
    \[ \Pr \left[ \cL^X_{n} ( n^{\gamma} ) \leq C n^{2 \gamma + \frac{1}{2\alpha} + \frac{3 \eps}{2}} \right]
    = \Pr \left[ \cL^X_{n} ( n^{\gamma} ) \leq C n^{\frac{1}{2} + \gamma +  \frac{7 \eps}{2}} \right]
    \geq 1 - \re^{-n^\eps} ,
    \]
    which yields the claimed result.
\end{proof}

\section*{Acknowledgements}
\addcontentsline{toc}{section}{Acknowledgements}

The authors are grateful to two anonymous referees for their comments and for indicating several errors or obscurities, and to Kilian Raschel for helpful discussions on computations of invariant measures, and references to the literature. 
The work of MM and AW was supported by EPSRC grant EP/W00657X/1.
SP was partially supported by
CMUP, member of LASI, which is financed by national funds
through FCT (Funda\c{c}\~ao
para a Ci\^encia e a Tecnologia, I.P.) 
under the project with reference UID/00144/2025. Part of this work was undertaken  during the programme ``Stochastic systems for anomalous diffusion'' (July--December 2024) hosted by the  Isaac Newton Institute, under EPSRC grant EP/Z000580/1.

\printbibliography

\end{document}